\documentclass[12pt,centertags,oneside,reqno]{amsart}
\usepackage{amsmath,amssymb}
\usepackage[colorlinks=true,linkcolor=magenta,citecolor=blue]{hyperref}
\usepackage[margin=1in]{geometry}
\usepackage[colorlinks=true,citecolor=blue]{hyperref}
\usepackage{enumitem} 
\usepackage{tikz}
\usetikzlibrary{arrows,automata,backgrounds,calendar,mindmap,trees}
\usepackage{inputenc}
\allowdisplaybreaks
\renewcommand{\varrho}{\mu}

\author[J. Morgan]{Jeff Morgan}
\address{Jeff Morgan \hfill\break
	Department of Mathematics, 
	University of Houston, Houston, Texas 77004, USA}
\email{jmorgan@math.uh.edu}
\author[B.Q. Tang]{Bao Quoc Tang}
\address{Bao Quoc Tang \hfill\break
	Institute of Mathematics and Scientific Computing, University of Graz, 
	Heinrichstrasse 36, 8010 Graz, Austria}
\email{quoc.tang@uni-graz.at, baotangquoc@gmail.com}

\title[Reaction-diffusion systems with intermediate sum conditions]{Boundedness for reaction-diffusion systems with Lyapunov functions and intermediate sum conditions}

\newcommand{\wh}{\widehat}

\newtheorem{theorem}{Theorem}[section]
\newtheorem{definition}{Definition}[section]
\newtheorem{lemma}{Lemma}[section]
\newtheorem{proposition}{Proposition}[section]

\newtheorem{remark}{Remark}[section]
\newtheorem{example}{Example}[section]

\begin{document}
	\begin{abstract}
		We study the uniform boundedness of solutions to reaction-diffusion systems possessing a Lyapunov-like function and satisfying an {\it intermediate sum condition}. This significantly generalizes the mass dissipation condition in the literature and thus allows the nonlinearities to have arbitrary polynomial growth. We show that two dimensional reaction-diffusion systems, with quadratic intermediate sum conditions, have global solutions which are bounded uniformly in time. In higher dimension, bounded solutions are obtained under the condition that the diffusion coefficients are {\it quasi-uniform}, i.e. they are close to each other. Applications include boundedness of solutions to chemical reaction networks with diffusion.
	\end{abstract}
%
%	\begin{center}
%		\Large{\bf Statement}
%	\end{center}
%		\vskip 1in
%		Dear Editors of the journal M3AS,
%	
%		\vskip 0.3in
%		This letter is to confirm that the present paper is not currently submitted to other journals and it will not be submitted to other journals during the reviewing process.
%		
%		\vskip 1in
%		\begin{flushright}
%			On the behalf of the authors\\
%			\medskip
%			{\it Bao Quoc Tang}
%		\end{flushright}
%	\newpage
	\maketitle
	
	\tableofcontents
	
	\section{Introduction and Main Results}	
	Let $\Omega\subset \mathbb R^n$, $n\geq 1$ be a bounded domain with smooth boundary $\partial\Omega$. In this paper, we study the global existence of classical solution to the following semilinear reaction-diffusion system
	\begin{equation}\label{e0}
	\begin{cases}
		\partial_t u_i - d_i\Delta u_i = f_i(u), &(x,t)\in \Omega\times(0,T),\\
		\nabla u_i\cdot \nu = 0, &(x,t)\in \partial\Omega\times(0,T),\\
		u_i(x,0) = u_{i,0}(x), &x\in\Omega,
	\end{cases}
	\end{equation}
	where $u = (u_1,\ldots, u_m)$, $d_i>0$ are diffusion coefficients, $\nu$ is the unit outward normal vector on $\partial\Omega$, $u_{i,0}$ are bounded, non-negative initial data, and the nonlinearities satisfy the following conditions:
	\begin{enumerate}[label=(A\theenumi),ref=A\theenumi]
		\item\label{A1} $f_i: \mathbb R^m \to \mathbb R$ is locally Lipschitz.
		\item\label{A2} For any $i=1,\ldots, m$, $f_i(u)\geq 0$ for all $u\in [0,\infty)^m$ provided $u_i = 0$.
		\item\label{A3} There exist nonnegative convex functions $h_i: [0,\infty) \to [0,\infty)$ which is continuous, and $h$ is $C^2$ on $(0,\infty)$, with $h_i'(z_i) > 0$ for some $z_i \in [0,\infty)$, and a nonnegative constant $K$ such that
		\begin{equation*}
			\sum_{i=1}^mh_i'(u)f_i(u) \leq K\left(1+\sum_{i=1}^mh_i(u)\right) \quad \text{ for all } \quad u\in (0,\infty)^m.
		\end{equation*}
		\item\label{A4} There exist an $m\times m$ lower triangle matrix $A = (a_{ij})$ with nonnegative elements and $a_{ii}>0$ for all $i=1,\ldots, m$, a constant $r\geq 1$ and a constant $C>0$ such that
		\begin{equation*}
			\sum_{j=1}^ka_{kj}f_j(u) \leq C\left(1+ \sum_{i=1}^mu_i\right)^r
		\end{equation*}
		for all $k=1,\ldots, m$ and $u\in [0,\infty)^m$.
		\item\label{A5} The nonlinearities are one-side bounded by a polynomial, i.e. there exist $C>0$ and $\varrho > 0$ such that
		\begin{equation*}
			f_i(u) \leq C\left(1+|u|^{\varrho}\right) \quad\text{ for all } \quad i=1,\ldots, m, \quad u\in [0,\infty)^m,
		\end{equation*}
		where $|u| = \sum_{j=1}^m|u_j|$.
	\end{enumerate}
	
	\medskip
	Let us comment on the assumptions \eqref{A1}--\eqref{A5}. The local Lipschitz continuity in \eqref{A1} is common in studying reaction-diffusion systems as it allows us to foremost obtain local existence. Assumption \eqref{A2} is relevant for systems arising from biology or chemistry, as it preserves the non-negativity of solutions, meaning that if the initial concentration (or density, population) is non-negative, then the solution is also non-negative as long as it exists. This assumption has a simple physical interpretation: if a concentration is zero, then it cannot be consumed in the reaction. The condition \eqref{A3} generalizes several common assumptions in the literature. Namely,
	\begin{itemize}
		\item {\it mass conservation}: when $h_i(u_i) = \alpha_i u_i$, $\alpha_i > 0$, $K=0$ and equality holds, i.e.
		\begin{equation}\label{mass_conservation}
			\sum_{i=1}^m \alpha_i f_i(u) = 0;
		\end{equation}
		\item {\it mass dissipation}: when $h_i(u_i) = \alpha_i u_i$, $\alpha_i > 0$, $K=0$, i.e.
		\begin{equation}\label{mass_dissipation}
			\sum_{i=1}^m \alpha_i f_i(u) \leq 0;
		\end{equation}
		\item {\it mass control}: when $h_i(u_i) = \alpha_i u_i$, $\alpha_i > 0$, $K>0$, i.e.
		\begin{equation}\label{mass_control}
			\sum_{i=1}^m\alpha_if_i(u) \leq K\left(1+\sum_{i=1}^m \alpha_i u_i\right);
		\end{equation}
		\item {\it entropy dissipation}: when $h_i(u_i) = u_i\log u_i - u_i+1$ and $K = 0$, i.e.
		\begin{equation}\label{entropy_inequality}
			\sum_{i=1}^m f_i(u)\log u_i \leq 0.
		\end{equation}
	\end{itemize}
	The condition \eqref{A5} indicates that the nonlinearities do not grow faster than polynomial of order $\mu$. We remark that our results in this paper will not have any restriction on the growth $\mu$. The assumption \eqref{A4} is called an {\it intermediate sum condition}, in the sense that only one of the nonlinearties is assumed to be one-side bounded by a polynomial of order $r$, while for the others we just need a good "cancellation". In practice, $r$ can be much smaller than $\mu$. To better explain that, we consider the following example which models the reversible reaction $pS_1 + qS_2 \leftrightarrows \ell S_3$, for $p, q \geq 1$,
	\begin{equation}\label{example}
	\begin{cases}
		\partial_t u_1 - d_1\Delta u_1 = f_1(u):= p(u_3^\ell - u_1^pu_2^q),\\
		\partial_t u_2 - d_2\Delta u_2 = f_2(u):= q(u_3^\ell - u_1^pu_2^q),\\
		\partial_t u_3 - d_3\Delta u_3 = f_3(u):= -\ell(u_3^\ell - u_1^pu_2^q).
	\end{cases}
	\end{equation}
	By choosing the matrix 
	\begin{equation*}
		A = \begin{pmatrix}
			1&0&0\\
			0&1&0\\
			q&p&2pq\ell^{-1}
		\end{pmatrix}
	\end{equation*}
	we can see that \eqref{A4} satisfies with $r = \ell$, while obviously $\varrho = \max\{p+q; \ell \}$ which can be much bigger than $r$ when $p$ and $q$ are large.

	\medskip
	With \eqref{A1}, \eqref{A2} and non-negative, bounded initial data, the local existence of a non-negative strong solution to \eqref{e0} on a maximal interval $(0,T_{\max})$ is classical (see e.g. \cite{Ama85}). The global existence of that local solution, on the other hand, is a challenging issue. Assuming for instance the mass dissipation \eqref{mass_dissipation}, one can easily show, using the homogeneous Neumann boundary conditions, that
	\begin{equation*}
		\frac{d}{dt}\sum_{i=1}^m\int_{\Omega}\alpha_iu_i(x,t)dx \leq 0.
	\end{equation*}
	Taking into account the non-negativity of the solution, this means that the solution is bounded in $L^\infty(0,T;L^1(\Omega))$ uniformly in $0<T<T_{\max}$. This is far from enough to obtain global existence. In fact, it was shown by a famous counterexample in \cite{PS00} that there exist systems satisfying \eqref{A1}--\eqref{A2} and \eqref{mass_dissipation} whose solutions blow up in $L^\infty$-norm in finite time. The global existence of a strong solution to \eqref{e0} with more structural conditions on the nonlinearities is therefore an interesting problem, and has been studied extensively in the literature. Let us first review some existing results in the literature, and from that highlight the novelty of our paper. Note that in the following, \eqref{A1} and \eqref{A2} are always assumed.
	\begin{itemize}
		\item In \cite{GV10}, assuming \eqref{mass_conservation} and \eqref{entropy_inequality}, it was shown that strong solutions exist globally if $n=1$ and $\mu = 3$ or $n=2$ and $\mu = 2$. The proof utilized the famous De Giorgi's method. This result was later reproved in \cite{Tan18a} by a modified Gagliardo-Nirenberg inequality.
		\item Still under \eqref{mass_conservation} and \eqref{entropy_inequality}, \cite{CV09} proved global existence for all $n\geq 1$ with strictly sub-quadratic growth, i.e. $\mu < 2$.
		\item By assuming only \eqref{mass_conservation}, \cite{CDF14} showed global existence of bounded solutions when $n=2$ and $\mu = 2$. Moreover, when $n\geq 2$ and $\mu \geq 2$, global strong solutions are proved under the {\it quasi-uniform diffusion} condition, i.e. the diffusion coefficients are close enough to each other. Similar results were also proved in \cite{FLS16}, and recently \cite{CMT}.
		\item By relaxing \eqref{mass_conservation} to \eqref{mass_dissipation}, \cite{PSY19} proved global existence when $n=2$ and $\mu = 2$. Therein, the solution is also proved to be bounded uniformly in time.
		\item In the close-to-equilibrium regime, it was shown in \cite{CC17} that if the initial data is close to equilibrium (in $L^2$-norm) then one can obtain global existence with $\mu = 2$ up to dimension $n=4$. This was later improved in \cite{Tan18} with the condition $\mu = 1+4/n$.
		\item The case $\mu = 2$ in higher dimensions had remained an open question until recently when it was settled in three different works \cite{CGV19, Sou18, FMT19}. The first work assumed \eqref{mass_conservation} and \eqref{entropy_inequality}, while the second relaxed \eqref{mass_conservation} to \eqref{mass_dissipation} but still needs \eqref{entropy_inequality}. The most general result is contained in the last work where the authors only assumed \eqref{mass_control}. The uniform-in-time bound has been shown recently in \cite{FMT19a}.  It's worth to mention the almost unnoticed work \cite{Kan90} where it considered $\Omega = \mathbb R^n$.
		\item Concerning a work assuming the general assumption \eqref{A3}, we refer to the works \cite{Mor90} and \cite{MW04}, in which the later showed the global existence with \eqref{mass_control} and quadratic intermediate sum condition but without uniform-in-time bounds.
		\item There are also number of works dealing with weaker notions of solutions. For instance it was shown in \cite{Pie03} under \eqref{mass_control} that if the nonlinearities belong to $L^1(\Omega\times(0,T))$ for any $T$, then one gets global weak solutions. By using a duality method, \cite{DFPV07} showed with \eqref{mass_control} global existence of weak solutions in all dimensions assuming $\mu = 2$. An even weaker notion called {\it renormalized solutions} was shown global in \cite{Fis15} assuming \eqref{entropy_inequality}. The interested reader is referred to the extensive survey \cite{Pie10} for more details.
	\end{itemize}
	It can be seen from the aforementioned works that the global existence of strong solutions with super-quadratic nonlinearities, i.e. $\mu > 2$, is much less studied except in \cite{CDF14} and \cite{FLS16}, where the diffusion coefficients are assumed to be quasi-uniform. That is the main motivation of our paper. More precisely, by utilizing the intermediate sum condition \eqref{A3}, we allow arbitrary polynomial growth of the nonlinearities.
	
	\medskip
	The first main result of this paper is the following.
	\begin{theorem}[Quadratic intermediate sums in two dimensions]\label{thm2}
			Assume that conditions \eqref{A1}--\eqref{A5} hold and let $n=2$ and $r = 2$ in \eqref{A4}. Then for any bounded, nonnegative initial data, \eqref{e0} has a unique nonnegative, global strong solution. 
			
			Moreover, if $K=0$ in \eqref{A3}, then the solution is bounded uniformly in time, i.e.
			\begin{equation}\label{ee}
			\sup_{t\geq 0}\|u_i(t)\|_{L^\infty(\Omega)} < +\infty \quad \text{ for all } \quad i=1,\ldots, m.
			\end{equation}
	\end{theorem}
	\begin{remark}[Improvements or generalizations]
	\hfill
	\begin{itemize}
		\item The assumption \eqref{A4} can be generalized to
				\begin{equation*}
				\sum_{j=1}^ka_{kj}g_j'(u_j)f_j(u) \leq C\left(1+\sum_{i=1}^mg_i(u_i)\right)^r, \quad \text{ for all } \quad k = 1,\ldots, m,
				\end{equation*}
				where $g_i: [0,\infty) \to [0,\infty)$ are $C^2$ nonnegative, convex functions satisfying $g_i'(z_i)> 0$ for some $z_i\in [0,\infty)$, provided there exists $M>0$ so that $g_i(z) \leq M(h_i(z) + 1)$ for $z$ sufficiently large, for all $i=1,\ldots, m$.
		\item The condition $r = 2$ can be slightly improved to $r = 2 + \varepsilon$ for sufficiently small $\varepsilon>0$. 
		\item In fact, to obtain the uniform bound \eqref{ee} as long as a uniform bound of the functions $h_i$ in $L^1$-norm is available, i.e.
		\begin{equation*}
			\sup_{i=1,\ldots, m}\sup_{t\geq 0}\|h_i(u_i)(t)\|_{L^1(\Omega)} < +\infty.
		\end{equation*}
		This bound follows straightforwardly from \eqref{A3} with $K = 0$ (see Lemma \ref{duality-lemma-cylinder}).
	\end{itemize}
	\end{remark}
	In Theorem \ref{thm2}, observe that the growth $\varrho$ in \eqref{A5} can be arbitrary. Which means that the nonlinearities can have arbitrarily high polynomial growth, as long as their intermediate sums (in the sense of \eqref{A4}) are bounded from the right by a quadratic polynomial. For instance, the system \eqref{example} with $\ell = 2$ satisfies the assumptions of Theorem \ref{thm2} and it therefore possesses a global strong solution, bounded uniformly in time, in two dimensions. 
	
	\medskip
	It's worthwhile to mention that the uniform-in-time bound of the solution is also of importance. This issue was usually left untouched in the literature, except for e.g. \cite{PSY19} or our recent work \cite{CMT}. In some situations arising from chemical reaction networks with boundary equilibria, the uniform $L^\infty$-bound plays a very important role. See more details in Examples \ref{ex1} and \ref{ex2}.
	
	\medskip
	Let us sketch the proof of Theorem \ref{thm2}, which can be roughly divided into several steps.
	\begin{description}
		\item[{\normalfont {\it Step 1}}] First, thanks to \eqref{A3}, we show that $u_i\in L^{p_0}(\Omega\times(0,T))$ for some $p_0>2$. This is proved thanks to an improved duality argument in \cite{CDF14}.
		\item[{\normalfont {\it Step 2}}] From this and \eqref{A4}, it can be proved that $u_1\in L^{q}(\Omega\times(0,T))$ for all $q<p_1$ where $p_1 = \frac{2p_0}{4-p_0}$.
		\item[{\normalfont {\it Step 3}}] This step is crucial where we prove that, for any $k=2,\ldots, m$, if $u_i \in L^{q}(\Omega\times(0,T))$ for all $q<p_1$, and all $i=1,\ldots, k-1$, then $u_k \in L^q(\Omega\times(0,T))$ for all $q<p_1$.
		\item[{\normalfont {\it Step 4}}] We then construct a sequence $\{p_N\}_{N\geq 0}$ with $p_{N+1} = \frac{2p_N}{4-p_N}$ for $p_N < 4$ and $p_{N+1}=\infty$ for $p_N\geq 4$, such that $u_i \in L^q(\Omega\times(0,T))$ for all $q<p_N$. 
		\item[{\normalfont {\it Step 5}}] Since $p_0>2$, the sequence $\{p_N\}$ is strictly increasing and there exists $p_{N_0}\geq 4$ which makes $p_{N_0} = +\infty$. Therefore, $u_i\in L^q(\Omega\times(0,T))$ for all $q<+\infty$. Finally, using the polynomial growth \eqref{A5}, one gets $u_i\in L^\infty(\Omega\times(0,T))$, hence the global existence of strong solutions.
		\item[{\normalfont {\it Step 6}}] To show the uniform boundedness in time, we repeat the above arguments but now in each cylinder $\Omega\times(\tau,\tau+1)$, $\tau \in \mathbb N$, and eventually show that $\|u_i\|_{L^\infty(\Omega\times(\tau,\tau+1))} \leq C$ where the constant $C$ is independent of $\tau$.
	\end{description}
	
	\medskip
	When $n\geq 3$ and $r>2$, the global existence of \eqref{e0} is largely open, except for the case when $h_i(u_i) = u_i$ and $\varrho = 2$, see e.g. \cite{CGV19,FMT19,Sou18}. In the next main result, we show that, if the diffusion coefficients are {\it quasi-uniform}, meaning that if they are close to each other, then one can obtain global strong solutions to \eqref{e0}.
	\begin{theorem}[Global existence and boundedness of strong solutions]\label{thm1}
		Assume that conditions \eqref{A1}--\eqref{A5} hold. Define 
		\begin{equation*}
			A = \min\{d_i: i=1,\ldots, m\} \quad \text{ and } \quad B = \max\{d_i: i=1,\ldots, m\}.
		\end{equation*}
		If
		\begin{equation}\label{quasi-uniform}
			C_{\frac{A+B}{2}, p'}<\frac{2}{B-A}
		\end{equation}
		where $p'>1$ such that
		\begin{equation*}
			p = \frac{p'}{p'-1} > \frac{n+2}{2}(r-1),
		\end{equation*}
		and $C_{\frac{A+B}{2},p'}$ is the constant defined in Lemma \ref{maximal-regularity} (a), then \eqref{e0} has a unique global nonnegative, bounded strong solution for any nonnegative, bounded initial data $u_{i0}$.

		\medskip		
		Moreover, if $K=0$ in \eqref{A3}, then the solution is bounded uniformly in time, i.e.
		\begin{equation*}
			\sup_{t\geq 0}\|u_i(t)\|_{L^\infty(\Omega)} < +\infty \quad \text{ for all } \quad i=1,\ldots, m.
		\end{equation*}
	\end{theorem}
	It is again remarked that the condition on the closeness of the diffusion coefficients \eqref{quasi-uniform} depends only on $r$ in \eqref{A4} and {\it independent of the polynomial growth} $\varrho$ in \eqref{A5}. This greatly improves related results in e.g. \cite{CDF14} or \cite{FLS16}.

	\begin{remark}
		In the recent work \cite{CMT}, global existence and uniform-in-time bounds were also obtained for \eqref{e0} under mass dissipation condition \eqref{mass_dissipation} and quasi-uniform diffusion coefficients. We remark that the latter condition imposed in \cite{CMT} {\it depends on the growth $\varrho$} of the nonlinearities, and therefore is much less general than \eqref{quasi-uniform}.
	\end{remark}	
	
	\medskip
	{\bf The rest of this paper is organized as follows}: For convenience, we split the proofs of Theorems \ref{thm2} and \ref{thm1} into two parts: global existence and uniform-in-time bounds, which will be proved in Sections \ref{global-existence} and \ref{uniform-bounds} respectively. To highlight the relevance of our results, we give some applications of our results to models arising from chemical reactions in Section \ref{applications}.
	
	\medskip
	{\bf Notation.} In this paper, we will use the following notation.
	\begin{itemize}
		\item We denote by $Q_{\tau,T} = \Omega\times(\tau,T)$. When $\tau = 0$, we write simply $Q_T = Q_{0,T}$. For any $1\leq p \leq \infty$, $L^p(Q_{\tau,T})$ stands for $L^p(\tau,T;L^p(\Omega))$.
		\item For any $T>0$, we write $C_T$ for a generic constant depending continuously on $T$, which can be different from line to line or even in the same line. More importantly, $C_T$ is defined for all $0<T<\infty$.
		\item For any $1<p\leq \infty$ we write
		\begin{equation*}
			\|u\|_{L^{p-}(Q_T)} \leq C_T \quad \text{ if } \quad \|u\|_{L^q(Q_T)} \leq C_{q,T} \text{ for all }1\leq q<p.
		\end{equation*}
		Here we write $C_{q,T}$ to indicate that the constant depends on $q$, and might blow up to infinity when $q \to p$.
	\end{itemize}
	\section{Global existence}\label{global-existence}
	\begin{definition}[Classical (or strong) solutions]
	 	Let $0<T\leq \infty$. A classical (or strong) solution to \eqref{e0} on $(0,T)$ is a vector of concentrations $u = (u_1, \ldots, u_m)$ satisfying for all $i=1,\ldots, m$, $u_i\in C([0,T); L^p(\Omega))\cap L^\infty(\Omega\times(0,\tau)) \cap C^2((0,T)\times\overline{\Omega})$ for all $p>n$ and all $0<\tau<T$, and $u$ solves \eqref{e0} pointwise.
	\end{definition}
	\begin{theorem}[Local existence of strong solutions]
		Assuming \eqref{A1}--\eqref{A2}. For any bounded, nonnegative initial data, \eqref{e0} possesses a local nonnegative strong solution on a maximal interval $[0,T_{\max})$. Moreover,
		\begin{equation}\label{blowup_criterion}
			\text{ if } \limsup_{t\uparrow T_{\max}}\|u_i(t)\|_{L^\infty(\Omega)}< \infty \text{ for all } i=1,\ldots, m \quad \text{ then } \quad T_{\max} = +\infty.
		\end{equation}
	\end{theorem}
	\begin{lemma}[Maximal regularity]\label{maximal-regularity}
		Let $0\leq \tau < T$, $m>0$ and $1<p<\infty$. Let $0\leq \theta \in L^p(Q_{\tau,T})$ and let $\phi$ be the unique nonnegative solution to
		\begin{equation}\label{dual-problem}
			\begin{cases}
				\partial_t \phi + m\Delta \phi = -\theta, &(x,t)\in Q_{\tau,T},\\
				\nabla \phi \cdot \nu = 0, &(x,t)\in \partial\Omega\times(\tau,T),\\
				\phi(x,T) = 0, &x\in\Omega.
			\end{cases}
		\end{equation}
		We have the following:
		\begin{itemize}
		\item[(a)] There exists a constant $C_{m,p}$ depending on $m$ and $p$, but {\normalfont independent of $\tau, T$} such that
		\begin{equation*}
			\|\Delta \phi\|_{L^p(Q_{\tau,T})} \leq C_{m,p}\|\theta\|_{L^p(Q_{\tau,T})}.
		\end{equation*}
		\item[(b)] 
		\begin{itemize}	
			\item[(i)] If $p\leq \frac{n+2}{2}$ then 
			\begin{equation*}
				\|\phi\|_{L^{p^*-}(Q_{\tau,T})} \leq C_{m,p,T-\tau}\|\theta\|_{L^p(Q_{\tau,T})} \quad \text{ with } \quad p^* = \frac{(n+2)p}{n+2 - 2p}.
			\end{equation*}
			\item[(ii)] If $p > \frac{n+2}{2}$ then 
			\begin{equation*}
				\|\phi\|_{L^\infty(Q_{\tau,T})} \leq C_{m,p,T-\tau}\|\theta\|_{L^p(Q_{\tau,T})}.
			\end{equation*}
		\end{itemize}
		\end{itemize}
	\end{lemma}
	\begin{proof}
		At first glance, \eqref{dual-problem} looks like a backward heat equation. However, with the change of variable $T+\tau - t$, we get the usual forward heat equation
		\begin{equation}\label{dual-problem-1}
			\begin{cases}
				\partial_t \varphi - m\Delta \varphi = \theta, &(x,t)\in Q_{\tau,T},\\
				\nabla \varphi \cdot \nu = 0, &(x,t)\in \partial\Omega\times(\tau,T),\\
				\varphi(x,\tau) = 0, &x\in\Omega.
			\end{cases}
		\end{equation}
		Therefore, part (a) can be found in \cite{Lam87} . Also from \cite{Lam87} we have
		\begin{equation*}
			\|\phi\|_{W^{(2,1)}_{p, Q_{\tau,T}}} \leq C_{m,p,T-\tau}\|\theta\|_{L^p(Q_{\tau,T})},
		\end{equation*}
		where $W^{(2,1)}_{p, Q_{\tau,T}}$ is defined via the norm
		\begin{equation*}
			\|f\|_{W^{(2,1)}_{p, Q_{\tau,T}}} = \sum_{2r+s\leq 2}\|\partial_t^r\partial_x^sf\|_{L^p(Q_{\tau,T})} <+\infty.
		\end{equation*}
		Part (b) then follows from this and the embeddings of space-time spaces, see e.g. \cite{LSU68}.
	\end{proof}
	\begin{lemma}\label{duality-lemma}
		Suppose the conditions of Theorem \ref{thm1} are satisfied. If
		\begin{equation}\label{quasi-uniform-1}
			C_{\frac{A+B}{2}, p'}<\frac{2}{B-A}
		\end{equation}
		for some $p'>1$. Then we have
		\begin{equation}\label{d0}
			\|h_i(u_i)\|_{L^p(Q_T)} \leq C_T \quad \text{ for all } \quad i=1,\ldots, m,
		\end{equation}
		where $p$ is the H\"older conjugate exponent of $p'$, i.e. $\frac 1p + \frac{1}{p'} = 1$. Consequently,
		\begin{equation}\label{d00}
			\|u_i\|_{L^p(Q_T)} \leq C_T \quad \text{ for all } \quad i=1,\ldots, m.
		\end{equation}
	\end{lemma}
	\begin{proof}
		We compute
		\begin{equation*}
			\begin{aligned}
			\partial_t\left(\sum_{i=1}^{m} h_i(u_i)\right) &= \sum_{i=1}^mh_i'(u_i)\partial_t u_i\\
			&= \sum_{i=1}^{m}h_i'(u_i)[d_i\Delta u_i + f_i(u)]\\
			&= \sum_{i=1}^m[d_i\Delta(h_i(u_i)) - d_ih_i''(u_i)|\nabla u_i|^2] + \sum_{i=1}^mh_i'(u_i)f_i(u)\\
			&\leq \Delta\left(\sum_{i=1}^{m}d_ih_i(u_i)\right) + K\left(1+\sum_{i=1}^mh_i(u_i) \right)
			\end{aligned}
		\end{equation*}
		where we used \eqref{A3} and $h_i''(u_i)\geq 0$, since $h_i$ is convex, at the last step. Define
		\begin{equation*}
			Z(x,t) = \sum_{i=1}^mh_i(u_i) \quad \text{ and } \quad M(x,t) = \frac{\sum_{i=1}^md_ih_i(u_i)}{Z}.
		\end{equation*}
		Then we have
		\begin{equation*}
			A \leq M(x,t) \leq B,
		\end{equation*}
		since $A = \min_{i=1,\ldots, m}d_i$ and $B = \max_{i=1,\ldots, m}d_i$. We therefore have
		\begin{equation}\label{d1}
			\partial_tZ \leq \Delta(MZ) + K(1+Z) \quad \text{ in } \quad Q_T,
		\end{equation}
		and $\nabla Z \cdot \nu = 0$ on $\partial\Omega\times(0,T)$. By integrating \eqref{d1} on $\Omega$ we get
		\begin{equation*}
			\partial_t\int_{\Omega}Z(x,t)dx \leq K\left(1+\int_{\Omega}Z(x,t)dx \right),
		\end{equation*}
		thus, thanks to Gronwall's lemma and the nonnegativity of $Z$,
		\begin{equation}\label{d1_1}
			\|Z(t)\|_{L^1(\Omega)} \leq Ce^{KT}\left(1 + \|Z(0)\|_{L^1(\Omega)}\right).
		\end{equation}
		Define $m = \frac{A+B}{2}$ and let $\phi$ be the nonnegative solution to
		\begin{equation*}
		\begin{cases}
			\partial_t \phi + m\Delta \phi = -\theta, &(x,t)\in Q_T,\\
			\nabla \phi \cdot \nu = 0, &(x,t)\in \partial\Omega\times(0,T),\\
			\phi(x,T) = 0, &x\in\Omega,\\
		\end{cases}
		\end{equation*}
		where $0\leq \theta \in L^{p'}(Q_T)$. Thanks to $\theta \geq 0$, we also have $\phi \geq 0$. 
		From Lemma \ref{maximal-regularity} we have
		\begin{equation}\label{d2}
			\|\Delta \phi\|_{L^{p'}(Q_T)} \leq C_{m,p'}\|\theta\|_{L^{p'}(Q_T)},
		\end{equation}
		\begin{equation}\label{d3}
			\|\partial_t\phi\|_{L^{p'}(Q_T)} + \|\phi\|_{L^s(Q_T)} \leq C_{m,p,T}\|\theta\|_{L^{p'}(Q_T)},
		\end{equation}
		for all $s < \frac{(n+2)p'}{n+2-2p'}$. By using integration by parts and \eqref{d1}, we compute
		\begin{equation}\label{d4}
			\begin{aligned}
				&\int_0^T\int_{\Omega}Z\theta dxdt\\
				&= -\int_0^T\int_{\Omega}Z(\partial_t \phi + m\Delta \phi)dxdt\\
				&= -\int_0^T\int_{\Omega}Z(\partial_t \phi + M\Delta \phi)dxdt + \int_0^T\int_{\Omega}Z(M-m)\Delta \phi dxdt\\
				&= \int_{\Omega}Z(0)\phi(0) +\int_0^T\int_{\Omega}\phi(\partial_t Z - \Delta(MZ))dxdt + \int_0^T\int_{\Omega}Z(M - m)\Delta \phi dxdt\\
				&\leq \int_{\Omega}Z(0)\phi(0) + K\int_0^T\int_{\Omega}\phi(1+Z)dxdt + \int_0^T\int_{\Omega}Z(M-m)\Delta \phi dxdt\\
				&\leq \|Z(0)\|_{L^p(\Omega)}\|\phi(0)\|_{L^{p'}(\Omega)} + K\|\phi\|_{L^s(Q_T)}\left(1+\|Z\|_{L^{s'}(Q_T)}\right)\\
				&\quad  + \frac{B-A}{2}\|Z\|_{L^p(Q_T)}\|\Delta \phi\|_{L^{p'}(Q_T)}
			\end{aligned}
		\end{equation}
		for some constant $s < \frac{(n+2)p'}{n+2-2p'}$ and $s' = \frac{s}{s-1}$. We will estimate the last three terms on the right-hand side of \eqref{d4} separately. First, using \eqref{d3},
		\begin{equation*}
			\|\phi(0)\|_{L^{p'}(\Omega)}^{p'} = \int_{\Omega}\left|\int_0^T\partial_t\phi(z)dz\right|^{p'}dx \leq C_T\|\partial_t\phi\|_{L^{p'}(Q_T)}^{p'} \leq C_T\|\theta\|_{L^{p'}(Q_T)}^{p'}.
		\end{equation*}
		Hence
		\begin{equation}\label{d5}
			\|Z(0)\|_{L^p(\Omega)}\|\phi(0)\|_{L^{p'}(\Omega)} \leq C_T\|Z(0)\|_{L^p(\Omega)}\|\theta\|_{L^{p'}(Q_T)}.
		\end{equation}
		By using \eqref{d3}, 
		\begin{equation}\label{d6}
			\frac{B-A}{2}\|Z\|_{L^p(Q_T)}\|\Delta \phi\|_{L^{p'}(Q_T)} \leq \frac{B-A}{2}C_{\frac{A+B}{2}, p'}\|Z\|_{L^p(Q_T)}\|\theta\|_{L^{p'}(Q_T)}.
		\end{equation}
		From \eqref{d3}, we choose $s = \frac{(n+2)p'}{n+3 - 2p'} < \frac{(n+2)p'}{n+2-2p'}$ and thus
		\begin{equation*}
			s' = \frac{s}{s-1} = \frac{(n+2)p}{n+3+p} < p.
		\end{equation*}
		Therefore, by interpolation inequality
		\begin{equation*}
			\|Z\|_{L^{s'}(Q_T)}^{s'} = \int_0^T\|Z\|_{L^{s'}(\Omega)}^{s'}dt \leq \int_0^T\|Z\|_{L^p(\Omega)}^{\alpha s'}\|Z\|_{L^1(\Omega)}^{(1-\alpha)s'}dt \leq C_T\|Z\|_{L^{p}(Q_T)}^{\alpha s'}
		\end{equation*}
		for some $\alpha\in (0,1)$, where we have used \eqref{d1_1} at the last step. Using this, we can estimate the second term on the right-hand side of \eqref{d4} as
		\begin{equation}\label{d7}
			K\|\phi\|_{L^{s}(Q_T)}\left(1+\|Z\|_{L^{s'}(Q_T)} \right) \leq C_T\|\theta\|_{L^{p'}(Q_T)}\left(1 + C_T\|Z\|_{L^p(Q_T)}^{\alpha}\right).
		\end{equation}
		Inserting \eqref{d5}, \eqref{d6} and \eqref{d7} into \eqref{d4}, we obtain
		\begin{equation*}
			\begin{aligned}
			\int_0^T\int_{\Omega}Z\theta dxdt &\leq C_T\|\theta\|_{L^{p'}(Q_T)}\left[\|Z(0)\|_{L^p(\Omega)}
			+ \|Z\|_{L^p(Q_T)}^{\alpha} \right] \\
			&\quad + \|\theta\|_{L^{p'}(Q_T)}\frac{B-A}{2}C_{\frac{A+B}{2}, p'}\|Z\|_{L^p(Q_T)} 
			\end{aligned}
		\end{equation*}
		for all $0\leq \theta \in L^{p'}(Q_T)$. By duality,
		\begin{equation*}
			\|Z\|_{L^p(Q_T)} \leq C_T\left[\|Z(0)\|_{L^p(\Omega)}  + \|Z\|_{L^p(Q_T)}^{\alpha} \right]+ \frac{B-A}{2}C_{\frac{A+B}{2}, p'}\|Z\|_{L^p(Q_T)},
		\end{equation*}
		and thanks to the assumption \eqref{quasi-uniform-1},
		\begin{equation*}
			\|Z\|_{L^p(Q_T)} \leq \left(1 - \frac{B-A}{2}C_{\frac{A+B}{2}, p'}\right)^{-1}C_T\left[\|Z(0)\|_{L^p(\Omega)}  + \|Z\|_{L^p(Q_T)}^{\alpha} \right].
		\end{equation*}
		Since $\alpha \in (0,1)$, we can use Young's inequality to finally obtain
		\begin{equation*}
			\|Z\|_{L^p(Q_T)} \leq C_T\left(1+\|Z(0)\|_{L^p(\Omega)}\right)
		\end{equation*}
		which implies \eqref{d0} due to $Z = \sum_{i=1}^m h_i(u_i)$ and the nonnegativity of $h_i(u_i)$.
		
		\medskip
		To get \eqref{d00}, we use the convexity of $h_i$ and the fact that $h_i'(z_i) > 0$ to obtain
		\begin{equation*}
			h_i(z) \geq h_i'(z_i)z + h_i(z_i) - z_ih_i'(z_i) \quad \text{ for all } \quad z\geq z_i.
		\end{equation*}
		Since $h_i$ is continuous on $[0,\infty)$, it is bounded on $[0,z_i]$ and therefore, we have
		\begin{equation}\label{d8}
			h_i(z) \geq h_i'(z_i)z - C \quad \text{ for all } \quad z \geq 0,
		\end{equation}
		for some constant $C>0$. It then follows from \eqref{d0} that
		\begin{equation*}
			\|u_i\|_{L^p(Q_T)} \leq C_T\left(1+\|Z(0)\|_{L^p(\Omega)}\right).
		\end{equation*}
	\end{proof}
	\begin{lemma}\label{Lp>2}
		Suppose the conditions of Theorem \ref{thm1} are satisfied. There exists $p>2$ such that 
		\begin{equation*}
			\|u_i\|_{L^p(Q_T)} \leq C_T \quad \text{ for all } \quad i=1,\ldots, m.
		\end{equation*}
	\end{lemma}
	\begin{proof}
		We only need to show that \eqref{quasi-uniform-1} holds for some $p' < 2$. Indeed, this can found in \cite[Lemma 3.2]{CDF14}, therefore the proof of Lemma \ref{Lp>2} follows from Lemma \ref{duality-lemma}.
	\end{proof}
	\begin{proposition}[Bootstrap technique]\label{bootstrap}
	Suppose the conditions of Theorem \ref{thm1} are satisfied.  
		If 
		\begin{equation*}
			\|u_i\|_{L^{p_0-}(Q_T)} \leq C_T \quad \text{ for all } \quad i=1,\ldots, m,
		\end{equation*}
		for some $p_0 > \frac{n+2}{2}(r-1)$, then 
		\begin{equation}\label{boundedness}
			\|u_i\|_{L^\infty(Q_T)} \leq C_T \quad \text{ for all} \quad i=1,\ldots, m.
		\end{equation}
	\end{proposition}
	\begin{proof}
		To prove Proposition \ref{bootstrap}, we construct a sequence $\{p_N\}_{N\geq 0}$ with
		\begin{equation}\label{p_N}
			p_{N+1} = \frac{(n+2)\frac{p_N}{r}}{n+2-2\frac{p_N}{r}}
		\end{equation}
		as long as $p_N \leq \frac{r(n+2)}{2}$, here we use the convention $\frac 10 = \infty$, such that
		\begin{equation*}
			\|u_i\|_{L^{p_{N+1}-}(Q_T)} \leq C_T
		\end{equation*}
		for all $N\geq 0$. We will do that in three steps.
		
		\medskip
		{\bf Step 1.} We claim that
		\begin{equation}\label{1claim}
			\text{ if } \sup_{i=1,\ldots, m}\|u_i\|_{L^{p_N-}(Q_T)}\leq C_T \quad \text{ then } u_1\in L^{p_{N+1}-}(Q_T).
		\end{equation}
		Let $q< p_{N+1}$ arbitrary. We choose $\wh{q}$ such that $q<\wh{q}<p_{N+1}$, and define
		\begin{equation}\label{e00}
			s: = \frac{(n+2)r\wh{q}}{n+2+2\wh{q}},\quad \text{ which is equivalent to }\quad \wh{q} = \frac{(n+2)\frac{s}{r}}{n+2-2\frac{s}{r}}.
		\end{equation}
		From $\wh{q}<p_{N+1}$ and \eqref{p_N}, it follows that $s<p_N$. Therefore,
		\begin{equation}\label{e1}
			\|u_i\|_{L^s(Q_T)} \leq C_T \quad \text{ for all } \quad i=1,\ldots, m.
		\end{equation}
		From the intermediate sum condition \eqref{A4}, we have
		\begin{equation}\label{e2}
			\partial_t(a_{11}u_1) - d_1\Delta(a_{11}u_1) = a_{11}f_1(u)\leq C\left(1+\sum_{i=1}^mu_i\right)^r=: g_1(u).
		\end{equation}
		From \eqref{e1}, the right hand side of \eqref{e2} satisfies
		\begin{equation*}
			\|g_1(u)\|_{L^{\frac{s}{r}}(Q_T)} \leq C_T.
		\end{equation*}
		Using comparison principle and maximal regularity in Lemma \ref{maximal-regularity} we have
		\begin{equation*}
			\|a_{11}u_1\|_{L^{\wh{q}-}(Q_T)} \leq C_T.
		\end{equation*}
		Recalling that $q<\wh{q}$, it follows
		\begin{equation*}
			\|u_1\|_{L^q(Q_T)} \leq C_T.
		\end{equation*}
		Since $q<p_{N+1}$ arbitrary, we obtain the claim \eqref{1claim}.
		
		\medskip
		{\bf Step 2.} We claim that, for any $k=2,\ldots, m$
		\begin{equation}\label{2claim}
		\begin{gathered}
			\text{ if } \sup_{i=1,\ldots, m}\|u_i\|_{L^{p_N-}(Q_T)} \leq C_T \text{ and } \sup_{j=1,\ldots,k-1}\|u_i\|_{L^{p_{N+1}-}(Q_T)}\leq C_T\\
			\text{ then } \|u_{k}\|_{L^{p_{N+1}-}(Q_T)}\leq C_T.
		\end{gathered}
		\end{equation}
		Fix $q<p_{N+1}$. Let $0\leq \theta \in L^{q'}(Q_T)$ with $q' = \frac{q}{q-1}$, and let $\phi$ be the nonnegative solution to
		\begin{equation}\label{dual}
		\begin{cases}
			\partial_t \phi + d_k\Delta \phi = -\theta, &(x,t)\in Q_T,\\
			\nabla\phi\cdot \nu = 0, &(x,t)\in \partial\Omega\times(0,T),\\
			\phi(x,T) = 0, &x\in\Omega.
		\end{cases}
		\end{equation}
		From Lemma \ref{maximal-regularity},
		\begin{equation}\label{e2_1}
			\|\partial_t \phi\|_{L^{q'}(Q_T)} + \|\Delta \phi\|_{L^{q'}(Q_T)}\leq C_T\|\theta\|_{L^{q'}(Q_T)}
		\end{equation}
		and
		\begin{equation}\label{e3}
			\|\phi\|_{L^{\gamma}(Q_T)} \leq C\|\theta\|_{L^{q'}(Q_T)} \quad \text{ for all } \quad \gamma < \frac{(n+2)q'}{n+2-2q'}.
		\end{equation}
		From the intermediate sums condition \eqref{A4} we have
		\begin{equation*}
			\partial_t\sum_{j=1}^ka_{kj}u_j - \sum_{j=1}^ka_{kj}d_j\Delta u_j = \sum_{j=1}^ka_{kj}f_i(u) \leq C\left(1+\sum_{i=1}^mu_i\right)^r,
		\end{equation*}
		thus
		\begin{equation}\label{e3_1}
			\partial_t(a_{kk}u_k) - d_k\Delta(a_{kk}u_k) \leq -\sum_{j=1}^{k-1}\left(\partial_t (a_{kj}u_j) - d_j\Delta (a_{kj}u_j) \right) + C\left(1+\sum_{i=1}^mu_i\right)^r.
		\end{equation}
		Using integration by parts we have
		\begin{equation}\label{e4}
		\begin{aligned}
			a_{kk}\int_0^T\int_{\Omega}u_k\theta dxdt &= -a_{kk}\int_0^T\int_{\Omega}u_k(\partial_t\phi + d_k\Delta \phi)dxdt\\
			&= a_{kk}\int_{\Omega}u_{k,0}\phi(0)dx + a_{kk}\int_0^T\int_{\Omega}\phi(\partial_t u_k - d_k\Delta u_k)dxdt\\
			&\leq a_{kk}\int_{\Omega}u_{k,0}\phi(0)dx - \int_0^T\int_{\Omega}\phi\sum_{j=1}^{k-1}(\partial_t(a_{kj}u_j - d_j\Delta(a_{kj}u_j)))dxdt\\
			&\quad + C\int_0^T\int_{\Omega}\phi\left(1+\sum_{i=1}^mu_i\right)^rdxdt\\
			&=\sum_{j=1}^ka_{kj}\int_{\Omega}u_{j,0}\phi(0)dx + \sum_{j=1}^{k-1}a_{kj}\int_0^T\int_{\Omega}u_j(\partial_t \phi - d_j\Delta\phi)dxdt\\
			&\quad + C\int_0^T\int_{\Omega}\phi dxdt + C\sum_{i=1}^m\int_0^T\int_{\Omega}\phi u_i^rdxdt\\
			&=: (I) + (II) + (III) + (IV).
		\end{aligned}
		\end{equation}
		We will estimate these four terms separately. Firstly, by H\"older's inequality
		\begin{equation*}
			|(I)| \leq C\sum_{i=1}^k\|u_{j,0}\|_{L^{q}(\Omega)}\|\phi(0)\|_{L^{q'}(\Omega)}.
		\end{equation*}
		Thanks to \eqref{e2_1} it follows that
		\begin{equation*}
			\|\phi(0)\|_{L^{q'}(\Omega)}^{q'} = \int_{\Omega}\left|\int_0^T\partial_t\phi(z)dz\right|^{q'}dx\leq C_T\|\partial_t\phi\|_{L^{q'}(Q_T)}^{q'},
		\end{equation*}
		and therefore
		\begin{equation}\label{A}
			|(I)| \leq C_T\sum_{j=1}^k\|u_{j,0}\|_{L^q(\Omega)}\|\theta\|_{L^{q'}(Q_T)}.
		\end{equation}
		For $(II)$, we use \eqref{dual}, \eqref{e2_1} and H\"older's inequality to estimate
		\begin{equation}\label{B}
		\begin{aligned}
			|(II)| &\leq \sum_{j=1}^{k-1}a_{kj}\int_0^T\int_{\Omega}|u_j||-\theta + (d_k - d_j)\Delta \phi|dxdt\\
			&\leq C\sum_{j=1}^{k-1}\|u_j\|_{L^{q}(Q_T)}\left[\|\theta\|_{L^{q'}(Q_T)} + \|\Delta\phi\|_{L^{q'}(Q_T)} \right]\\
			&\leq C\sum_{j=1}^{k-1}\|u_j\|_{L^q(Q_T)}\|\theta\|_{L^{q'}(Q_T)}.
		\end{aligned}
		\end{equation}
		The term $(III)$ is easily estimated that
		\begin{equation}\label{C}
			|(III)| \leq C_T\|\phi\|_{L^{p'}(Q_T)}\leq C_T\|\theta\|_{L^{p'}(Q_T)}.
		\end{equation}
		Since $q < p_{N+1}$ and $q' = \frac{q}{q-1}$,
		\begin{equation*}
			\frac{(n+2)q'}{n+2-2q'} = \frac{(n+2)q}{nq - (n+2)} > \frac{(n+2)p_{N+1}}{np_{N+1} - (n+2)} = \frac{p_N}{p_N - r}
		\end{equation*}
		where we used \eqref{p_N} at the last step. From this we can choose $s$ such that
		\begin{equation}\label{def_s}
			\frac{(n+2)q'}{n+2 - 2q'} > s > \frac{p_N}{p_N - r},
		\end{equation}
		and it follows from \eqref{e3} that
		\begin{equation*}
			\|\phi\|_{L^s(Q_T)}\leq C\|\theta\|_{L^{q'}(Q_T)}.
		\end{equation*}
		We now estimate the $(IV)$ as, with $s' = \frac{s}{s-1}$,
		\begin{equation}\label{D}
		\begin{aligned}
			|(IV)| &\leq C\sum_{i=1}^m\|\phi\|_{L^s(Q_T)}\|u_i\|_{L^{rs'}(Q_T)}^r\\
			&\leq C\sum_{i=1}^m\|\theta\|_{L^{q'}(Q_T)}\|u_i\|_{L^{rs'}(Q_T)}^r\\
			&\leq C\sum_{i=1}^m\|\theta\|_{L^{q'}(Q_T)}C_T
		\end{aligned}
		\end{equation}
		where at the last step we used $rs' = \frac{rs}{s-1} < p_N$, thanks to \eqref{def_s}, and thus
		\begin{equation*}
			\|u_i\|_{L^{rs'}(Q_T)}\leq C_T \quad \text{ for all } \quad i=1,\ldots, m.
		\end{equation*}
		Inserting \eqref{A}, \eqref{B}, \eqref{C} and \eqref{D} into \eqref{e4}, we obtain
		\begin{equation*}
			a_{kk}\int_0^T\int_{\Omega}u_k\theta dxdt \leq C_T\left[\sum_{j=1}^m\|u_{j,0}\|_{L^q(\Omega)}+ \sum_{j=1}^{k-1}\|u_j\|_{L^q(Q_T)} +1\right]\|\theta\|_{L^{q'}(Q_T)},
		\end{equation*}
		for all $0\leq \theta \in L^{q'}(Q_T)$, and therefore, by duality
		\begin{equation*}
			\|u_k\|_{L^{q}(Q_T)} \leq C_T\left[\sum_{j=1}^m\|u_{j,0}\|_{L^q(\Omega)}+ \sum_{j=1}^{k-1}\|u_j\|_{L^q(Q_T)} +1\right],
		\end{equation*}
		for $q<p_{N+1}$ arbitrary, which proves the claim \eqref{2claim}.
		
		\medskip
		From the claims \eqref{1claim} and \eqref{2claim} we now have for all $N\geq 0$
		\begin{equation*}
			\|u_i\|_{L^{p_N-}(Q_T)} \leq C_T, \quad \text{ for all } \quad i=1,\ldots, m
		\end{equation*}
		as long as $p_N \leq \frac{r(n+2)}{2}$. Since $p_0 > \frac{n+2}{2}(r-1)$, it's easy to check from \eqref{p_N} 
		\begin{equation*}
			\frac{p_{N+1}}{p_N} > \frac{n+2}{r(n+2) - 2p_0} > 1.
		\end{equation*}
		Therefore, there exists $N_0>0$ such that 
		\begin{equation*}
			p_{N_0} \geq \frac{r(n+2)}{2}.
		\end{equation*}
		From this and {\bf Step 1.} and {\bf Step 2.} we get
		\begin{equation}\label{e5}
			\|u_i\|_{L^{\infty-}(Q_T)}\leq C_T
		\end{equation}
		for all $i=1,\ldots, m$. To finally obtain the $L^\infty$-bound of the solution, we use \eqref{e5} and the assumption \eqref{A5} that all $f_i(u)$ are one-side bounded by a polynomial. Indeed, from \eqref{A5} and the comparison principle we have
		\begin{equation*}
			u_i(x,t) \leq v_i(x,t) \quad \text{ a.e. } \quad \text{ in } Q_T,
		\end{equation*}
		where $(v_1, \ldots, v_m)$ is the solution to
		\begin{equation*}
		\begin{cases}
			\partial_t v_i - d_i\Delta v_i = C\left(1 + \sum_{j=1}^mu_j^{\mu} \right), &(x,t)\in Q_T,\\
			\nabla v_i \cdot \nu = 0, &(x,t)\in \partial\Omega\times (0,T),\\
			v_i(x,0) = u_{i,0}(x), &x\in\Omega.
		\end{cases}
		\end{equation*}
		Thanks to \eqref{e5},
		we can apply Lemma \ref{maximal-regularity} to obtain $\|v_i\|_{L^\infty(Q_T)} \leq C_T$ and consequently
		\begin{equation*}
			\|u_i\|_{L^\infty(Q_T)} \leq C_T
		\end{equation*}
		which finishes the proof of Lemma \ref{bootstrap}.
	\end{proof}
	
	\medskip
	We are now ready to prove the global existence parts in Theorems \ref{thm1} and \ref{thm2}.
	
		\medskip
		\begin{proof}[Proof of global existence in Theorem \ref{thm2}]
			The global existence of strong solutions follows directly from Lemma \ref{Lp>2} and Proposition \ref{bootstrap}, by checking the blow-up criterion \eqref{blowup_criterion}. To show uniqueness, we assume that $u^{(1)} = (u^{(1)}_i)$ and $u^{(2)}  = (u^{(2)}_i)$ are two strong solutions on $(0,T)$ for arbitrary $T>0$, with the same initial data. By multiplying 
			\begin{equation*}
				\partial_t (u^{(1)}_i - u^{(2)}_i) - d_i\Delta (u^{(1)}_i - u^{(2)}_i) = f_i(u^{(1)}) - f_i(u^{(2)})
			\end{equation*}
			by $u^{(1)}_i - u^{(2)}_i$ the integrating on $\Omega$, we get
			\begin{equation}\label{uni1}
				\frac 12 \frac{d}{dt}\|u^{(1)}_i - u^{(2)}_i\|_{L^2(\Omega)}^2 + d_i\|\nabla (u^{(1)}_i - u^{(2)}_i)\|_{L^2(\Omega)}^2 = \int_{\Omega}(f_i(u^{(1)}) - f_i(u^{(2)}))(u^{(1)}_i - u^{(2)}_i)dx.
			\end{equation}
			Since $\|u^{(1)}_i\|_{L^\infty(Q_T)}, \|u^{(2)}_i\|_{L^\infty(Q_T)} \leq C_T$, and $f_i$ are locally Lipschitz continuous (see \eqref{A1}),
			\begin{align*}
				\left|\int_{\Omega}(f_i(u^{(1)}) - f_i(u^{(2)}))(u^{(1)}_i - u^{(2)}_i)dx\right| &\leq C_T\int_{\Omega}|u^{(1)} - u^{(2)}||u^{(1)}_i - u^{(2)}_i|dx\\
				&\leq C_T\sum_{j=1}^m\|u^{(1)}_j - u^{(2)}_j\|_{L^2(\Omega)}^2.
			\end{align*}
			Inserting this into \eqref{uni1} then summing up with respect to $i=1,\ldots, m$, we obtain in particular
			\begin{equation*}
				\frac{d}{dt}\sum_{i=1}^m\|u^{(1)}_i - u^{(2)}_i\|_{L^2(\Omega)}^2 \leq C_T\sum_{i=1}^m\|u^{(1)}_i - u^{(2)}_i\|_{L^2(\Omega)}^2
			\end{equation*}
			Using Gronwall's inequality and the fact that $u^{(1)}(0) = u^{(2)}(0)$ we finally have
			\begin{equation*}
				\|u^{(1)}_i(t) - u^{(2)}_i(t)\|_{L^2(\Omega)} = 0 \quad \text{ for all } \quad t\in (0,T) \quad \text{ and all } \quad i=1,\ldots, m,
			\end{equation*}
			which proves the uniqueness of strong solutions.
		\end{proof}	
	
	\medskip
	\begin{proof}[Proof of global existence in Theorem \ref{thm1}]
		From the assumption \eqref{quasi-uniform} and Lemma \ref{duality-lemma} we have
		\begin{equation*}
			\|u_i\|_{L^p(Q_T)} \leq C_T
		\end{equation*}
		for $p > \frac{n+2}{2}(r-1)$. Thanks to Proposition \ref{bootstrap} we get the bound \eqref{boundedness}, and therefore the global existence of strong solution. The uniqueness is proved similarly to Theorem \ref{thm2} and therefore omitted.
	\end{proof}
	\section{Uniform-in-time bounds}\label{uniform-bounds}
	From Section \ref{global-existence}, we know that the strong solution exists globally. This section aims to show that if $K = 0$ in \eqref{A3}, then the solution is bounded uniformly in time. The idea is to study \eqref{e0} on each cylinder $\Omega\times(\tau,\tau+1)$, $\tau \geq 0$, and show that the solution is bounded on this cylinder {\it independently of $\tau$}. 
	
	\medskip
	In this section, we always denote by $C$ a generic constant which is {\it independent of $\tau$}, which can be different from line to line, or even in the same line.
	
	\medskip
	We will need the following elementary lemma, whose proof is straightforward.
	\begin{lemma}\label{elementary}
		Let $\{y_n\}_{n\geq 0}$ be a nonnegative sequence. Define $\mathcal N = \{n\in \mathbb N: y_n \leq y_{n+1}\}$. If there exists $C>0$ such that
		\begin{equation*}
			y_n \leq C \quad \text{ for all } \quad n \in \mathcal N,
		\end{equation*}
		then
		\begin{equation*}
			y_n \leq \max\{y_0, C\} \quad \text{ for all } \quad n\in \mathbb N.
		\end{equation*}
	\end{lemma}
	\begin{lemma}\label{duality-lemma-cylinder}
		Suppose the conditions in Theorem \ref{thm1} are satisfied.
		Assume \eqref{quasi-uniform-1} for some $p'>1$. Then there exists $C>0$ such that for any $\tau \geq 0$,
		\begin{equation}\label{f0}
			\|u_i\|_{L^p(\Omega\times(\tau,\tau+1))} \leq C \quad \text{ for all } \quad i=1,\ldots, m.
		\end{equation}
	\end{lemma}
	\begin{proof}
		Similar to Lemma \ref{duality-lemma}, we have the following
		\begin{equation}\label{f1}
			\partial_t Z \leq \Delta(MZ)
		\end{equation}
		with $\nabla Z \cdot \nu = 0$ on $\partial\Omega\times(0,\infty)$.
		where
		\begin{equation*}
			Z(x,t) = \sum_{i=1}^mh_i(u_i) \quad \text{ and } \quad M(x,t) = \frac{\sum_{i=1}^md_ih_i(u_i)}{Z}
		\end{equation*}
		Integrating \eqref{f1} on $(0,t)$ yields
		\begin{equation}\label{f2}
			\|Z(t)\|_{L^1(\Omega)} = \int_{\Omega}Z(x,t)dx \leq \int_{\Omega}Z(x,0)dx\leq \|Z(0)\|_{L^1(\Omega)} \quad \text{ for all } \quad t\geq 0.
		\end{equation}
		We define for $\tau \geq 0$ a smooth function $\varphi_\tau: [0,\infty) \to [0,1]$ such that $\varphi_\tau(\tau) = 0$, $\varphi_\tau(t) \equiv 1$ for all $t\geq \tau+1$, and $|\varphi_\tau'| \leq C$. It then follows from \eqref{f1} that
		\begin{equation}\label{f3}
			\partial_t(\varphi_\tau Z) \leq \Delta(M\varphi_\tau Z) + \varphi_\tau' Z, \quad \text{ and } \quad (\varphi_\tau Z)(x,\tau) =0.
		\end{equation}
		
		\medskip
		Define $m = \frac{A+B}{2}$ and let $\phi$ be the nonnegative solution to
		\begin{equation*}
			\begin{cases}
				\partial_t \phi + m\Delta \phi = -\theta, &(x,t)\in Q_{\tau,\tau+2},\\
				\nabla \phi \cdot \nu = 0, &(x,t)\in \partial\Omega\times(\tau,\tau+2),\\
				\phi(x,\tau+2) = 0, &x\in\Omega
			\end{cases}
		\end{equation*}
		where $0 \leq \theta \in L^{p'}(Q_{\tau,\tau+2})$. From Lemma \ref{maximal-regularity} we have
		\begin{equation}\label{f3_1}
			\|\Delta \phi\|_{L^{p'}(Q_{\tau,\tau+2})} \leq C_{m,p'}\|\theta\|_{L^{p'}(Q_{\tau,\tau+2})},
		\end{equation}
		\begin{equation}\label{f3_2}
			\|\phi\|_{L^s(Q_{\tau,\tau+2})} \leq C\|\theta\|_{L^{p'}(Q_{\tau,\tau+2})}
		\end{equation}
		for all $s < \frac{(n+2)p'}{n+2 - 2p'}$. By using integration by parts
		\begin{equation}\label{f4}
		\begin{aligned}
			&\int_{\tau}^{\tau+2}\int_{\Omega}(\varphi_\tau Z) \theta dxdt\\
			&=-\int_{\tau}^{\tau+2}\int_{\Omega}(\varphi_\tau Z)(\partial_t\phi + m\Delta \phi)dxdt\\
			&= \int_{\tau}^{\tau+2}\int_{\Omega}\phi\left(\partial_t(\varphi_\tau Z) - \Delta(M\varphi_\tau Z)\right) + \int_{\tau}^{\tau+2}\int_{\Omega}(M-m)(\varphi_\tau Z) \Delta \phi dxdt\\
			&\leq \int_{\tau}^{\tau+2}\int_{\Omega}\phi \varphi_\tau' Z dxdt + \int_{\tau}^{\tau+2}\int_{\Omega}(M-m)(\varphi_\tau Z) \Delta \phi dxdt.
		\end{aligned}
		\end{equation}
		For the last term on the right-hand side of \eqref{f4}, we use \eqref{f3_1} to estimate
		\begin{equation}\label{f5}
		\begin{aligned}
			\int_{\tau}^{\tau+2}\int_{\Omega}(M-m)(\varphi_\tau Z) \Delta \phi dxdt &\leq \frac{B-A}{2}\|\varphi_\tau Z\|_{L^p(Q_{\tau,\tau+2})}\|\Delta \phi\|_{L^{p'}(Q_{\tau,\tau+2})}\\
			&\leq \frac{B-A}{2}C_{\frac{A+B}{2}, p'}\|\varphi_\tau Z\|_{L^{p}(Q_{\tau,\tau+2})}\|\theta\|_{L^{p'}(Q_{\tau,\tau+2})}.
		\end{aligned}
		\end{equation}
		From \eqref{f3_2}, we choose $s = \frac{(n+2)p'}{n+2 - 2p'} < \frac{(n+2)p'}{n+2-2p'}$ and thus
		\begin{equation*}
		s' = \frac{s}{s-1} = \frac{(n+2)p}{n+3+p} < p.
		\end{equation*}
		Therefore, by interpolation inequality
		\begin{equation*}
			\|Z\|_{L^{s'}(Q_{\tau,\tau+2})}^{s'} = \int_{\tau}^{\tau+2}\|Z\|_{L^{s'}(\Omega)}^{s'}dt \leq \int_{\tau}^{\tau+2}\|Z\|_{L^p(\Omega)}^{\alpha s'}\|Z\|_{L^1(\Omega)}^{(1-\alpha)s'}dt \leq C\|Z\|_{L^p(Q_{\tau,\tau+2})}^{\alpha s'}
		\end{equation*}
		where we used \eqref{f2}at the last step. Now we can estimate
		\begin{equation}\label{f6}
			\int_{\tau}^{\tau+2}\int_{\Omega}\phi \varphi_\tau' Z dxdt \leq C\|\phi\|_{L^s(Q_{\tau,\tau+2})}\|Z\|_{L^{s'}(Q_{\tau,\tau+2})} \leq C\|Z\|_{L^p(Q_{\tau,\tau+2})}^{\alpha}\|\theta\|_{L^{p'}(Q_{\tau,\tau+2})}.
		\end{equation}
		Inserting \eqref{f5} and \eqref{f6} into \eqref{f4} we get
		\begin{equation*}
		\begin{aligned}
			\int_{\tau}^{\tau+2}\int_{\Omega}(\varphi_\tau Z)\theta dxdt &\leq C\|Z\|_{L^p(Q_{\tau,\tau+2})}^\alpha\|\theta\|_{L^{p'}(Q_{\tau,\tau+2})}\\
			&\quad + \frac{B-A}{2}C_{\frac{A+B}{2}, p'}\|\varphi_\tau Z\|_{L^p(Q_{\tau,\tau+2})}\|\theta\|_{L^{p'}(Q_{\tau,\tau+2})},
		\end{aligned}
		\end{equation*}
		for all $0\leq \theta\in L^{p'}(Q_{\tau,\tau+2})$, and thus, by duality and \eqref{quasi-uniform-1}
		\begin{equation}\label{f6_1}
			\|\varphi_\tau Z\|_{L^p(Q_{\tau,\tau+2})} \leq C\left(1-\frac{B-A}{2}C_{\frac{A+B}{2}, p'}\right)^{-1}\|Z\|_{L^p(Q_{\tau,\tau+2})}^{\alpha} \leq C\|Z\|_{L^p(Q_{\tau,\tau+2})}^{\alpha}.
		\end{equation}
		Consider all $\tau \in \mathbb N$ such that
		\begin{equation}\label{f7}
			\|Z\|_{L^p(Q_{\tau,\tau+1})} \leq \|Z\|_{L^p(Q_{\tau+1,\tau+2})}.
		\end{equation}
		It follows that
		\begin{equation*}
			\|Z\|_{L^p(Q_{\tau,\tau+2})} \leq 2^{1/p}\|Z\|_{L^p(Q_{\tau+1,\tau+2})}.
		\end{equation*}
		By using $\varphi_\tau \geq 0$ and $\varphi_\tau \equiv 1$ on $[\tau+1,\tau+2]$ we get from \eqref{f6_1}
		\begin{equation*}
			\|Z\|_{L^p(Q_{\tau+1,\tau+2})} \leq \|\varphi_\tau Z\|_{L^p(Q_{\tau,\tau+2})} \leq C2^{1/p}\|Z\|_{L^p(Q_{\tau+1,\tau+2})}^\alpha.
		\end{equation*}
		Since $\alpha \in (0,1)$, we can use Young's inequality to get finally 
		\begin{equation}\label{f8}
			\|Z\|_{L^p(Q_{\tau+1,\tau+2})} \leq C
		\end{equation} 
		for all $\tau \in \mathbb N$ satisfying \eqref{f7}. Lemma \ref{elementary} therefore implies that \eqref{f8} is true for all $\tau \in \mathbb N$. Thus,
		\begin{equation*}
			\|h_i(u_i)\|_{L^p(Q_{\tau+1,\tau+2})} \leq C\quad \text{ for all } \quad \tau\in \mathbb N,
		\end{equation*}
		hence \eqref{f0} thanks to \eqref{d8}.
	\end{proof}
	
	\begin{proposition}\label{bootstrap-cylinder}
		If
		\begin{equation*}
			\|u_i\|_{L^{p_0}-(Q_{\tau,\tau+1})} \leq C \quad \text{ for all } \quad \tau \geq 0, \; i=1,\ldots, m
		\end{equation*}
		for some $p_0 > \frac{n+2}{2}(r-1)$, then
		\begin{equation}\label{boundedness-cylinder}
			\|u_i\|_{L^{\infty}(Q_{\tau,\tau+1})} \leq C, \quad \text{ for all } \quad \tau \geq 0, \; i=1,\ldots, m.
		\end{equation}
	\end{proposition}
	\begin{proof}
		We will use similar ideas to Proposition \ref{bootstrap} on the cylinder $Q_{\tau,\tau+2} = \Omega\times(\tau,\tau+2)$.
		
		Let $\{p_N\}$ be defined as in \eqref{p_N}, we will show that
		\begin{equation*}
			\|u_i\|_{L^{p_N-}(Q_{\tau,\tau+2})} \leq C \quad \text{ for all } \quad i=1,\ldots, m, \; N\geq 0.
		\end{equation*}
		
		\medskip
		{\bf Step 1.} We claim that
		\begin{equation}\label{1claim-cylinder}
			\text{ if } \quad \sup_{i=1,\ldots, m}\|u_i\|_{L^{p_N-}(Q_{\tau,\tau+2})} \leq C \quad \text{ then } \quad u_1 \in L^{p_{N+1}-}(Q_{\tau,\tau+2}).
		\end{equation}
		Let $q<p_{N+1}$ arbitrary. We choose $\wh{q}$ such that $q<\wh{q}<p_{N+1}$ and define $s$ as in \eqref{e00}. Since $\wh{q} < p_{N+1}$ and \eqref{p_N}, $s<p_N$. Thus
		\begin{equation}\label{g1}
			\|u_i\|_{L^s(Q_{\tau,\tau+2})} \leq C.
		\end{equation}
		Let $\varphi_\tau$ be defined as in Lemma \ref{duality-lemma-cylinder}, we have, thanks to \eqref{A4},
		\begin{equation*}
			\partial_t(a_{11}\varphi_\tau u_1) - d_1\Delta(a_{11}\varphi_\tau u_1) = a_{11}\varphi_\tau' u_1 + a_{11}f_1(u) \leq a_{11}\varphi_{\tau}' u_1 + C\left(1+\sum_{i=1}^m u_i\right)^r =: k_1(u).
		\end{equation*}
		From \eqref{g1}, 
		\begin{equation*}
			\|k_1(u)\|_{L^{\frac{s}{r}}(Q_{\tau,\tau+2})} \leq C.
		\end{equation*}
		Using the comparison principle and maximal regularity in Lemma \ref{maximal-regularity} we have
		\begin{equation*}
			\|a_{11}\varphi_\tau u_1\|_{L^{\wh{q}-}(Q_{\tau,\tau+2})} \leq C
		\end{equation*}
		which implies
		\begin{equation*}
			\|u_1\|_{L^{q}(Q_{\tau+1,\tau+2})} \leq C
		\end{equation*}
		since $q<\wh{q}$ and $\varphi_\tau \equiv 1$ on $[\tau+1,\tau+2]$. Since $q< p_{N+1}$ arbitrary, this implies \eqref{1claim-cylinder}.
		
		\medskip
		{\bf Step 2.} We claim that, for any $k=2, \ldots, m$,
		\begin{equation}\label{2claim-cylinder}
		\begin{gathered}
			\text{ if } \quad \sup_{i=1,\ldots, m}\|u_i\|_{L^{p_N-}(Q_{\tau,\tau+2})} \leq C \quad \text{ and } \quad \sup_{j=1,\ldots, k-1}\|u_i\|_{L^{p_{N+1}-}(Q_{\tau,\tau+2})} \leq C,\\
			\text{ then } \|u_k\|_{L^{p_{N+1}-}(Q_{\tau,\tau+2})} \leq C.
		\end{gathered}
		\end{equation}
		Fix $q<p_{N+1}$. Let $q' = \frac{q}{q-1}$, $0\leq \theta \in L^{q'}(Q_{\tau,\tau+2})$, and let $\phi$ be the nonnegative solution to
		\begin{equation*}
			\begin{cases}
				\partial_t \phi + d_k\Delta \phi = -\theta, &(x,t)\in Q_{\tau,\tau+2},\\
				\nabla \phi \cdot \nu = 0, &(x,t)\in \partial\Omega\times(\tau,\tau+2),\\
				\phi(x,\tau+2) = 0, &x\in\Omega.
			\end{cases}
		\end{equation*}
		From Lemma \ref{maximal-regularity},
		\begin{equation}\label{g2}
			\|\phi\|_{L^{q'}(Q_{\tau,\tau+2})} + \|\Delta \phi\|_{L^{q'}(Q_{\tau,\tau+2})} \leq C\|\theta\|_{L^{q'}(Q_{\tau,\tau+2})}
		\end{equation}
		and
		\begin{equation}\label{g3}
			\|\phi\|_{L^{\gamma}(Q_{\tau,\tau+2})} \leq C\|\theta\|_{L^{q'}(Q_{\tau,\tau+2})} \quad \text{ for all } \quad \gamma < \frac{(n+2)q'}{n+2-2q'}.
		\end{equation}
		From \eqref{e3_1}, 
		\begin{equation*}
		\begin{aligned}
			&\partial_t(a_{kk}\varphi_\tau u_k) - d_k\Delta(a_{kk}\varphi_\tau u_k)\\
			&\leq -\sum_{j=1}^{k-1}\left[\partial_t(a_{kj}\varphi_\tau u_j)- d_j\Delta(a_{kj}\varphi_\tau u_j)\right] + \varphi_\tau'\sum_{j=1}^ku_j + C\varphi_\tau\left(1+\sum_{i=1}^m u_i\right)^r.
		\end{aligned}
		\end{equation*}
		Therefore, by using integration by parts, we have
%		\begin{equation}\label{g4}
		\begin{align}\label{g4}
			&a_{kk}\int_{\tau}^{\tau+2}\int_{\Omega}(\varphi_\tau u_k)\theta dxdt\nonumber\\
			&= -a_{kk}\int_{\tau}^{\tau+2}\int_{\Omega} \varphi_\tau u_k(\partial_t \phi + d_k\Delta \phi)dxdt\nonumber\\
			&= \int_{\tau}^{\tau+2}\int_{\Omega}\phi(\partial_t(a_{kk}\varphi_\tau u_k) - d_k\Delta(a_{kk}\varphi_\tau u_k))dxdt\nonumber\\
			&\leq -\sum_{j=1}^{k-1}a_{kj}\int_{\tau}^{\tau+2}\int_{\Omega}\phi(\partial_t(\varphi_\tau u_j) - d_j\Delta(\varphi_\tau u_j))dxdt + \varphi_\tau' \sum_{j=1}^k\int_{\tau}^{\tau+2}\int_{\Omega} \phi u_j dxdt\nonumber\\
			&\quad + C\int_{\tau}^{\tau+2}\int_{\Omega} \phi\varphi_\tau\left(1+\sum_{i=1}^m u_i\right)^rdxdt\nonumber\\
			&\leq \sum_{j=1}^{k-1}\int_{\tau}^{\tau+2}\int_{\Omega}\varphi_\tau u_j[-\theta + (d_j-d_k)\Delta \phi]dxdt + \varphi_\tau' \sum_{j=1}^k\int_{\tau}^{\tau+2}\int_{\Omega} \phi u_j dxdt\nonumber\\
			&\quad + C\int_{\tau}^{\tau+2}\int_{\Omega}\phi dxdt + C\sum_{i=1}^m\int_{\tau}^{\tau+2}\int_{\Omega}\phi u_i^r dxdt\\
			&=: (I) + (II) + (III) + (IV)\nonumber.
		\end{align}
%		\end{equation}
		For $(I)$ we use H\"older's inequality and \eqref{g2} to estimate
		\begin{equation}\label{I}
			\begin{aligned}
			|(I)| &\leq C\sum_{j=1}^{k-1}\|u_j\|_{L^{q}(Q_{\tau,\tau+2})}\left(\|\theta\|_{L^{q'}(Q_{\tau,\tau+2})} + |d_j-d_k|\|\Delta\phi\|_{L^{q'}(Q_{\tau,\tau+2})}\right)\\
			&\leq C\sum_{j=1}^{k-1}\|u_j\|_{L^{q}(Q_{\tau,\tau+2})}\|\theta\|_{L^{q'}(Q_{\tau,\tau+2})}\\
			&\leq C\|\theta\|_{L^{q'}(Q_{\tau,\tau+2})},
			\end{aligned}
		\end{equation}
		where we used $\|u_j\|_{L^q(Q_{\tau,\tau+2})} \leq C$ for all $j=1,\ldots, k-1$. 	For $(II)$ we first choose $\gamma = \frac{(n+2)q'}{n+3 - 2q'} < \frac{(n+2)q'}{n+2-2q'}$, which implies $\|\phi\|_{L^\gamma(Q_{\tau,\tau+2})}\leq C\|\theta\|_{L^{q'}(Q_{\tau,\tau+2})}$. Therefore, with $\gamma' = \frac{\gamma}{\gamma - 1}$,
		\begin{equation*}
		\begin{aligned}
			|(II)| &\leq C\sum_{j=1}^{k-1}\|\phi\|_{L^{q'}(Q_{\tau,\tau+2})}\|u_j\|_{L^q(Q_{\tau,\tau+2})} + C\|\phi\|_{L^\gamma(Q_{\tau,\tau+2})}\|u_k\|_{L^{\gamma'}(Q_{\tau,\tau+2})}\\
			&\leq C\|\theta\|_{L^{q'}(Q_{\tau,\tau+2})} + C\|\theta\|_{L^{q'}(Q_{\tau,\tau+2})}\|u_k\|_{L^{\gamma'}(Q_{\tau,\tau+2})}
		\end{aligned}
		\end{equation*}
		where we used $\|u_j\|_{L^q(Q_{\tau,\tau+2})} \leq C$ for all $j=1,\ldots, k-1$. 	Note that
		\begin{equation*}
			\gamma' = \frac{\gamma}{\gamma - 1} = \frac{(n+2)q}{n+3+q} < q.
		\end{equation*}
		Thus, we can use the interpolation inequality
		\begin{equation*}
			\|u_k\|_{L^{\gamma'}(Q_{\tau,\tau+2})}^{\gamma'} = \int_{\tau}^{\tau+2}\|u_k\|_{L^{\gamma'}(\Omega)}^{\gamma'}dt \leq \int_{\tau}^{\tau+2}\|u_k\|_{L^q(\Omega)}^{\alpha\gamma'}\|u_k\|_{L^1(\Omega)}^{(1-\alpha)\gamma'dt} \leq C\|u_k\|_{L^q(Q_{\tau,\tau+2})}^{\alpha \gamma'}
		\end{equation*}
		where we used $\|u_k(t)\|_{L^1(\Omega)} \leq C$, which is implied from \eqref{f2} and \eqref{d8}. We can then estimate $(II)$ further as
		\begin{equation}\label{II}
			|(II)| \leq C\|\theta\|_{L^{q'}(Q_{\tau,\tau+2})} + C\|\theta\|_{L^{q'}(Q_{\tau,\tau+2})}\|u_k\|_{L^q(Q_{\tau,\tau+2})}^{\alpha}.
		\end{equation}
		The term $(III)$ can be easily estimated as
		\begin{equation}\label{III}
			|(III)| \leq C\|\phi\|_{L^{q'}(Q_{\tau,\tau+2})} \leq C\|\theta\|_{L^{q'}(Q_{\tau,\tau+2})}.
		\end{equation}
		Since $q<p_{N+1}$, we choose $s$ such that \eqref{def_s} holds, and thus, thanks to \eqref{g3},
		\begin{equation*}
			\|\phi\|_{L^s(Q_{\tau,\tau+2})} \leq C\|\theta\|_{L^{q'}(Q_{\tau,\tau+2})}.
		\end{equation*}
		We now can estimate the last term $(IV)$ as
		\begin{equation}\label{IV}
		\begin{aligned}
			|(IV)| &\leq C\sum_{i=1}^m\|\phi\|_{L^{s}(Q_{\tau,\tau+2})}\|u_i\|_{L^{rs'}(Q_{\tau,\tau+2})}^r\\
			&\leq C\sum_{i=1}^{m}\|\theta\|_{L^{q'}(Q_{\tau,\tau+2})}
		\end{aligned}
		\end{equation}
		where we used at the last step
		\begin{equation*}
			\|u_i\|_{L^{rs'}(Q_{\tau,\tau+2})} \leq C \quad \text{ for all } \quad i=1,\ldots, m,
		\end{equation*}
		thanks to $rs' = \frac{rs}{s-1} < p_N$. From \eqref{I}, \eqref{II}, \eqref{III} and \eqref{IV} we obtain from \eqref{g4} that
		\begin{equation*}
			a_{kk}\int_{\tau}^{\tau+2}\int_{\Omega}(\varphi_\tau u_k)\theta dxdt \leq C\|\theta\|_{L^{q'}(Q_{\tau,\tau+2})}\left(1+ \|u_k\|_{L^q(Q_{\tau,\tau+2})}^\alpha \right)
		\end{equation*}
		for all $0\leq \theta \in L^{q'}(Q_{\tau,\tau+2})$. Therefore, by duality,
		\begin{equation*}
			\|\varphi_\tau u_k\|_{L^q(Q_{\tau,\tau+2})} \leq C + C\|u_k\|_{L^q(Q_{\tau,\tau+2})}^{\alpha}.
		\end{equation*}
		Consider $\tau\in \mathbb N$ such that
		\begin{equation}\label{g5}
			\|u_k\|_{L^q(Q_{\tau,\tau+1})} \leq \|u_k\|_{L^q(Q_{\tau+1,\tau+2})}.
		\end{equation}
		Then, using $\varphi_\tau \equiv 1$ on $[\tau+1,\tau+2]$ we get
		\begin{equation*}
			\|u_k\|_{L^q(Q_{\tau+1,\tau+2})} \leq \|\varphi_\tau u_k\|_{L^q(Q_{\tau,\tau+2})} \leq C + C2^{\alpha/p}\|u_k\|_{L^{q}(Q_{\tau+1,\tau+2})}^\alpha,
		\end{equation*}
		which implies, thanks to Young's inequality and $\alpha \in (0,1)$,
		\begin{equation*}
			\|u_k\|_{L^q(Q_{\tau+1,\tau+2})} \leq C
		\end{equation*}
		for all $\tau$ such that \eqref{g5} holds. Thanks to Lemma \ref{elementary}, this inequality in fact holds for all $\tau \in \mathbb N$, which proves the claim \eqref{2claim-cylinder} since $q<p_{N+1}$ arbitrary.
		
		\medskip
		Thanks to the claims \eqref{1claim-cylinder} and \eqref{2claim-cylinder}, we can proceed exactly as the last part of the proof of Proposition \ref{bootstrap} to see that there exists $N_0$ such that
		\begin{equation*}
			p_{N_0} \geq \frac{r(n+2)}{2}.
		\end{equation*}
		Thus, for all $\tau \geq 0$,
		\begin{equation}\label{ff}
			\|u_i\|_{L^{\infty-}(Q_{\tau,\tau+1})} \leq C \quad \text{ for all } \quad i=1,\ldots, m.
		\end{equation}
		Using this and the fact that
		\begin{equation*}
		\begin{cases}
			\partial_t(\varphi_\tau u_i) - d_i\Delta(\varphi_\tau u_i) = \varphi_\tau' u_i + \varphi_\tau f_i(u) =: g_i(u), &(x,t)\in Q_{\tau,\tau+2},\\
			\nabla(\varphi_\tau u_i)\cdot \nu = 0, &(x,t)\in \partial\Omega\times (\tau,\tau+2),\\
			(\varphi_\tau u_i)(x,\tau) = 0, &x\in\Omega,
		\end{cases}
		\end{equation*}
		with, thanks to \eqref{A5}, 
		\begin{equation*}
			g_i(u) \leq \varphi_\tau' u_i + C\varphi_\tau\left(1+\sum_{j=1}^mu_j^{\mu} \right) =: h_i(u).
		\end{equation*}
		Due to \eqref{ff}, we have
		\begin{equation*}
			\|h\|_{L^q(Q_{\tau,\tau+2})} \leq C \quad \text{ for all } \quad 1\leq q < +\infty.
		\end{equation*}
		Therefore, the comparison principle and Lemma \ref{maximal-regularity} gives $\|\varphi_\tau u_i\|_{L^\infty(Q_{\tau,\tau+2})} \leq C$ and consequently
		\begin{equation*}
			\|u_i\|_{L^\infty(Q_{\tau,\tau+1})} \leq C,
		\end{equation*}
		which finishes the proof of Proposition \ref{bootstrap-cylinder}.
	\end{proof}	

	\medskip
	\begin{proof}[Proof of uniform bounds in Theorem \ref{thm2}]
		From \cite[Lemma 3.2]{CDF14}, there exists $p'< 2$ such that
		\begin{equation*}
			\frac{B-A}{2}C_{\frac{A+B}{2}, p'} < 1.
		\end{equation*}
		Therefore, it follows from Lemma \ref{duality-lemma-cylinder} that $\|u_i\|_{L^p(Q_{\tau,\tau+1})} \leq C$ for all $i=1,\ldots, m$ and all $\tau\geq 0$. The uniform boundedness in time the follows directly from Proposition \ref{bootstrap-cylinder}.
	\end{proof}	
		
	\medskip
	\begin{proof}[Proof of uniform bounds in Theorem \ref{thm1}]
		From the assumption \eqref{quasi-uniform} and Lemma \ref{duality-lemma-cylinder} we have
		\begin{equation*}
			\|u_i\|_{L^p(Q_{\tau,\tau+1})} \leq C \quad \text{ for all } \quad i=1,\ldots, m, \; \tau\geq 0,
		\end{equation*}
		for $p > \frac{n+2}{2}(r-1)$. Thanks to Proposition \ref{bootstrap-cylinder} we get the bound \eqref{boundedness-cylinder}, and therefore the uniform-in-time bound of the solution.
	\end{proof}
		
	\section{Applications}\label{applications}
	In this section, we provide five examples to demonstrate applications of our main results to reaction-diffusion systems arising from chemical reaction networks. We emphasize that the uniform bounds in time that we get play an important role in studying the large time behavior of these systems.
	
	\begin{example}\label{ex1}
	\end{example}
	\vskip -0.1in
		In \cite{CJPT}, the author considered the reversible reaction $A+kB \overset{k_2}{\underset{k_1}{\leftrightarrows}} B + C$ for some $2\leq k\in \mathbb N$, which results in the system
		\begin{equation}\label{S1}
			\begin{cases}
				\partial_t a - d_a\Delta a = -k_1 ab^k + k_2bc, &(x,t)\in Q_{\infty},\\
				\partial_t b - d_b\Delta b = -k_1 ab^k + k_2bc, &(x,t)\in Q_{\infty},\\
				\partial_t c - d_c\Delta c = k_1ab^k - k_2bc, &(x,t)\in Q_{\infty},\\
				\nabla a\cdot \nu = \nabla b\cdot \nu = \nabla c\cdot \nu = 0, &(x,t)\in \partial\Omega\times(0,\infty),\\
				a(\cdot, 0) = a_0, \; b(\cdot, 0) = b_0, \; c(\cdot, 0) = c_0, &x\in\Omega,
			\end{cases}
		\end{equation}
		where $Q_\infty = \Omega\times(0,\infty)$, $a(x,t)$, $b(x,t)$ and $c(x,t)$ are the concentration densities of $A, B$ and $C$, respectively.
		It was shown in \cite{CJPT} that if $n=1$ and if $a_0, b_0, c_0\in L^\infty(\Omega)$ are nonnegative and 
		\begin{equation*}
			\int_{\Omega}a_0(x)dx\cdot\int_{\Omega}b_0(x)dx\cdot\int_{\Omega}c_0(x)dx > 0 \quad \text{ and } \quad b_0(x) \geq \delta \text{ a. e. in } \Omega,
		\end{equation*}
		for some $\delta > 0$, then \eqref{S1} has a unique global strong solution, which converges exponentially to the unique positive equilibrium (see \cite[Theorem 1.1.]{CJPT}).
		
		\medskip
		The idea of \cite{CJPT} is to first show that the solution to \eqref{S1} is bounded uniformly in time, then an entropy-entropy dissipation technique leads to the exponential decay to equilibrium. While the latter part is dimensional independent, they explicitly need $n=1$ to get the uniform-in-time bound of the solution. Therefore, they can only obtain the decay to equilibrium in one dimension.
		
		Looking at system \eqref{S1}, we can easily see that it satisfies our assumptions \eqref{A1}--\eqref{A5} with $h_1(a) = a$, $h_2(b) = b$ and $h_3(c) = 2c$, $K=0$, $r=2$, $\varrho = k+1$ and the matrix
		\begin{equation*}
			A = \begin{pmatrix}
				1&0&0\\
				1&1&0\\
				1&1&2
			\end{pmatrix}.
		\end{equation*}
		It is remarked also that in this case, the assumption \eqref{A3} becomes the usual mass dissipation condition \eqref{mass_conservation}. 	Our Theorem \ref{thm2} is therefore applicable and consequently, we obtain the convergence result in \cite{CJPT} also in two dimensions.
		\begin{theorem}[Exponential convergence to equilibrium for \eqref{S1} in two dimensions]
		Let $n=2$ and assume that $a_0, b_0, c_0\in L^\infty(\Omega)$ are nonnegative with 
		\begin{equation*}
			\int_{\Omega}a_0(x)dx\cdot\int_{\Omega}b_0(x)dx\cdot\int_{\Omega}c_0(x)dx > 0 \quad \text{ and } \quad b_0(x) \geq \delta \text{ a.e. in } \Omega,
		\end{equation*}		
		for some $\delta > 0$. Then \eqref{S1} has a unique global strong solution, which is bounded uniformly in time, i.e.
		\begin{equation}\label{bound}
			\sup_{t\geq 0}\left[\|a(t)\|_{L^\infty(\Omega)} + \|b(t)\|_{L^\infty(\Omega)} + \|c(t)\|_{L^\infty(\Omega)} \right] < +\infty.
		\end{equation}
		Moreover, as $t\to \infty$, the solution converges exponentially fast to the unique positive equilibrium in $L^p$-norm for any $1\leq p<\infty$, i.e. 
		\begin{equation}\label{Lp}
			\|a(t) - a_\infty\|_{L^p(\Omega)} + \|b(t) - b_\infty\|_{L^p(\Omega)}  + \|c(t) - c_\infty\|_{L^p(\Omega)} \leq Ce^{-\lambda_p t} \quad \text{ for all } \quad t\geq 0,
		\end{equation}
		where $C, \lambda_p > 0$, and $(a_\infty, b_\infty, c_\infty)$ is the positive solution to 
		\begin{equation*}
			\begin{cases}
				-k_1a_\infty b_\infty^k  + k_2b_\infty c_\infty = 0,\\
				a_\infty + c_\infty = \frac{1}{|\Omega|}\left(\int_{\Omega}a_0dx + \int_{\Omega}c_0dx \right),\\
				b_\infty + c_\infty = \frac{1}{|\Omega|}\left(\int_{\Omega}b_0dx + \int_{\Omega}c_0dx \right).
			\end{cases}
		\end{equation*}
		\end{theorem}
		\begin{proof}
			The global existence and uniform boundedness follow directly from Theorem \ref{thm2}. The exponential equilibration in $L^1$-norm was shown in \cite{CJPT}, i.e.
			\begin{equation*}
						\|a(t) - a_\infty\|_{L^1(\Omega)} + \|b(t) - b_\infty\|_{L^1(\Omega)}  + \|c(t) - c_\infty\|_{L^1(\Omega)} \leq Ce^{-\lambda_1 t} \quad \text{ for all } \quad t\geq 0,
					\end{equation*}
			for some $C, \lambda_1 > 0$. The convergence in $L^p$-norm \eqref{Lp} then follows from the uniform $L^\infty$-bound \eqref{bound} and the interpolation inequality
		\begin{equation*}
			\|f\|_{L^p(\Omega)} \leq \|f\|_{L^\infty(\Omega)}^{\frac{p-1}{p}}\|f\|_{L^1(\Omega)}^{\frac 1p}.
		\end{equation*}			
		\end{proof}
	
	\begin{example}\label{ex2}
	\end{example}
	\vskip -0.1in
		We consider in this example the following reactions
		\begin{equation}\label{reaction2}
			A_1 + \cdots + A_m + kA_{m+1} \leftrightarrows A_m + A_{m+1} \leftrightarrows B_1 + \cdots + B_h + \ell A_m
		\end{equation}
		where $k, \ell \geq 2$. For simplicity, we assume that all the reaction rate constants are equal to one. Using the law of mass action, we obtain the following system
		\begin{equation}\label{S2}
		\partial_t\begin{bmatrix}
		u_1\\
		\vdots\\
		u_{m-1}\\
		u_m\\
		u_{m+1}\\
		v_1\\
		\vdots\\
		v_h
		\end{bmatrix}
		- D\Delta\begin{bmatrix}
				u_1\\
				\vdots\\
				u_{m-1}\\
				u_m\\
				u_{m+1}\\
				v_1\\
				\vdots\\
				v_h
				\end{bmatrix}
		= \begin{bmatrix}
			-u_1u_2\cdots u_mu_{m+1}^k + u_mu_{m+1}\\
			\vdots\\
			-u_1u_2\cdots u_mu_{m+1}^k + u_mu_{m+1}\\
			-(\ell-1)u_m^{\ell}v_1v_2\cdots v_h + (\ell-1)u_mu_{m+1}\\
			-(k-1)u_1\cdots u_mu_{m+1}^k + (k-2)u_mu_{m+1} + u_m^{\ell}v_1\cdots v_h\\
			u_mu_{m+1}-u_m^{\ell}v_1v_2\cdots v_h\\
			\vdots\\
			u_mu_{m+1}-u_m^{\ell}v_1v_2\cdots v_h
		\end{bmatrix}
		\end{equation}
		with homogeneous Neumann boundary conditions
		\begin{equation*}
			\nabla u_i \cdot \nu = \nabla v_j \cdot \nu = 0, \quad \text{  for } \quad (x,t)\in \partial\Omega\times(0,\infty)
		\end{equation*}
		and initial data
		\begin{equation*}
			u_i(x,0) = u_{i,0}(x), \qquad v_j(x,0) = v_{j,0}(x), \quad \text{ for } \quad x\in\Omega,
		\end{equation*}
		for all $i=1,\ldots, m$ and $j = 1,\ldots, h$. Here we denote by $u_i$ and $v_j$ the concentration densities of $A_i$ and $B_j$, respectively. System \eqref{S2} was left in \cite{CJPT} as an open problem, both for the global existence as well as the large time behavior.
		
		\medskip
		For simplicity, we write $u = (u_1, \ldots, u_m, u_{m+1}, v_1, \ldots, v_h)$ and $f_i(u)$, $i=1,\ldots, m+1$, the nonlinearity for the equation of $u_i$, and $f_{m+1+j}(u)$, $j=1,\ldots, h$, the nonlinearity for the equation of $v_j$. By straightforward calculations, one can show that there {\it do not} exist positive constants $\alpha_1, \ldots, \alpha_{m+1+h}$ such that
		\begin{equation*}
			\sum_{i=1}^{m+1+h}\alpha_if_i(u) \leq C\left(1 + \sum_{i=1}^{m+1}u_i + \sum_{j=1}^h v_j\right).
		\end{equation*}
		Therefore, the existing works using the mass control condition \eqref{mass_control} are not applicable. Luckily, due to the reversibility nature of \eqref{reaction2}, we have a dissipative structure of the entropy. More precisely, denote by $h_i(z) = z\log z - z + 1$, we can easily see that
		\begin{equation*}
		\begin{aligned}
			\sum_{i=1}^{m+1+h}h_i'(u_i)f_i(u) &= -[u_1u_2\cdots u_mu_{m+1}^k - u_mu_{m+1}]\log\frac{u_1u_2\cdots u_mu_{m+1}^k}{u_mu_{m+1}}\\
			&\quad - [u_mu_{m+1} - u_{m}^\ell v_1v_2\cdots v_h]\log\frac{u_mu_{m+1}}{u_{m}^\ell v_1v_2\cdots v_h} \leq 0,
		\end{aligned}
		\end{equation*}
		and thus \eqref{A3} is satisfied. With direct computations, we see that \eqref{A1}, \eqref{A2}, \eqref{A4} and \eqref{A5} are also satisfied with $r = 2$, $\varrho = \max\{m+k, \ell + h\}$ and the matrix 
		\begin{equation*}
			A = \begin{pmatrix}
			I_{m-1}&0&0&0\\
			0&1&0&0\\
			0&1&(\ell-1)&0\\
			0&0&\overrightarrow{1}^\top&I_h
			\end{pmatrix}
		\end{equation*}
		where $\overrightarrow{1}^\top$ is a column of size $h$ with all elements are one, and $I_{m-1}$ and $I_h$ are identities matrices of size $(m-1)$ and $h$, respectively. Therefore, we can apply Theorem \ref{thm2} to get global, uniform-in-time bounded solutions to \eqref{S2} in two dimensions.
		\begin{theorem}
			Let $n=2$. Then for any bounded, nonnegative initial data $u_{i0}, v_{j0}\in L^{\infty}(\Omega)$, \eqref{S2} has a unique, bounded strong solution with
			\begin{equation*}
				\sup_{t\geq 0}\sup_{i,j}\left[\|u_i(t)\|_{L^\infty(\Omega)} + \|v_j(t)\|_{L^\infty(\Omega)}\right] < +\infty.
			\end{equation*}
		\end{theorem}
		We believe this uniform bound will help to obtain convergence to equilibrium for \eqref{S2}. We leave this for future investigation.
	
	\begin{example}\label{ex3}
	\end{example}
	\vskip -0.1in
	In the last example, we consider the following reaction network
	
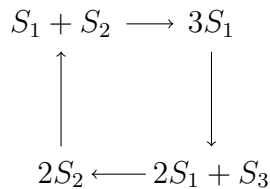
\begin{equation*}
	\begin{tikzpicture} [baseline=(current  bounding  box.center)]
	\node (a) {$S_1+S_2$} node (b) at (2,0) {$3S_1$} node (c) at (2,-2) {$2S_1+S_3$} node (d) at (0,-2) {$2S_2$};
	
	\draw[arrows=->] ([xshift =0.5mm]a.east) -- ([xshift =-0.5mm]b.west);
	
	\draw[arrows=->] ([yshift=0.5mm]d.north) -- ([yshift=-0.5mm]a.south);
	
	\draw[arrows=->] ([yshift=-0.5mm]b.south) -- ([yshift =0.5mm]c.north);
	
	\draw[arrows=->] ([xshift =0.5mm]c.west) -- ([xshift =-0.5mm]d.east);
	\end{tikzpicture}
	\end{equation*}
	whose ODE setting was studied in \cite{And11}. Denote $u_i(x,t)$ as the concentration densities of $S_i$, we obtain the following system thanks to the law of mass action
	\begin{equation}\label{S3}
		\begin{cases}
			\partial_tu_1 - d_1\Delta u_1 = 2u_1u_2 - u_1^3 - 2u_1^2u_3 + u_2^2=: f_1(u), &(x,t)\in Q_T,\\
			\partial_t u_2 - d_2\Delta u_2 = -u_1u_2 + 2u_1^2u_3 - u_2^2=: f_2(u), &(x,t)\in Q_T,\\
			\partial_t u_3 - d_3\Delta u_3 = u_1^3 - u_1^2u_3=: f_3(u), &(x,t)\in Q_T
		\end{cases}
	\end{equation}	
	with homogeneous Neumann boundary conditions $\nabla u_i \cdot \nu = 0$ on $\partial\Omega\times(0,T)$ and nonnegative initial data $u_{i}(x,0) = u_{i,0}(x)$ on $\Omega$. We can again easily check that the assumptions \eqref{A1}, \eqref{A2}, \eqref{A4} and \eqref{A5} are satisfied with $r = 2$, $\varrho = 3$, and the matrix
	\begin{equation*}
		A = \begin{pmatrix}
			1&0&0\\
			1&1&0\\
			1&1&1
		\end{pmatrix}.
	\end{equation*}
	Similarly to Example \ref{ex2}, there {\it do not} exist positive constants $\alpha_i, i=1,\ldots, 3$ such that \eqref{mass_control} holds. To justify \eqref{A3}, we use again an entropic dissipative structure of \eqref{S3}. More precisely, by using $h_i(u_i) = u_i\log u_i - u_i + 1$ we can show with direct computations that
	\begin{equation*}
		\begin{aligned}
		\sum_{i=1}^3h_i'(u_i)f_i(u) = - \Phi(u_1u_2, u_1^3) - \Phi(u_1^3, u_1^2u_3) - \Phi(u_1^2u_3, u_2^2) - \Phi(u_2^2, u_1u_2)\leq 0
		\end{aligned}
	\end{equation*}
	where $\Phi(x,y) = x\log(x/y) - x + y$. Therefore, we have the following
	\begin{theorem}
			Let $n=2$. Then for any bounded, nonnegative initial data $u_{i0}\in L^{\infty}(\Omega)$, \eqref{S3} has a unique, bounded strong solution with
			\begin{equation*}
				\sup_{t\geq 0}\sup_{i=1,\ldots, 3}\|u_i(t)\|_{L^\infty(\Omega)} < +\infty.
			\end{equation*}
	\end{theorem}
	The convergence to equilibrium for the reaction-diffusion system \eqref{S3} is, up to our knowledge, completely open.
		
	\begin{example}\label{ex4}
	\end{example}
	\vskip -.1in
	The last example deals with the system \eqref{example} in the Introduction. This type of system was considered in \cite{Laa11} in which the following cases were treated:
	\begin{enumerate}
	\item $p + q < \ell$;
	\item ($d_1=d_3$ or $d_2=d_3$) and for any $(p,q,\ell)$;
	\item $d_1=d_2$ and for any $(p,q,\ell)$ such that $p+q \ne \ell$;
	\item $1<\ell < \frac{n+6}{n+2}$ and for any $(p,q)$.
	\end{enumerate}
	The results of our paper allow us to deal with the case $\ell = 2$, $n=2$ and arbitrary $(p,q)$ and arbitrary diffusion coefficients $d_i>0$, which is not included in \cite{Laa11}.
	\begin{theorem}
		Let $n=2$ and $\ell = 2$. Then for any bounded, non-negative initial data, \eqref{example} has a unique global strong solution, which is bounded uniformly in time, i.e.
		\begin{equation*}
			\sup_{i=1,\ldots, 3}\sup_{t\geq 0}\|u_i(t)\|_{L^\infty(\Omega)} <+\infty.
		\end{equation*}
	\end{theorem}

	\begin{example}\label{ex5}
	\end{example}
	\vskip -0.1in
	
	Our last example considers the following system considered in \cite{FT18},
	\begin{equation}\label{example1}
		\begin{cases}
			\partial_t u_1 - d_1\Delta u_1 = u_3 - u_2^qu_1,\\
			\partial_t u_2 - d_2\Delta u_2 = qu_3 + u_2^q u_1 - (q+1)u_2^{q+1},\\
			\partial_t u_3 - d_3\Delta u_3 = -u_3 + u_2^{q+1},
		\end{cases}
	\end{equation}
	where $q\geq 1$, subject to homogeneous Neumann boundary conditions and bounded, non-negative initial data $u_i(x,0) = u_{i,0}(x)$ for $x\in\Omega$ and $i=1,\ldots, 3$. System \eqref{example1} models the weakly reversible reaction network $S_1 + qS_2 \to (q+1)S_2 \to S_3 \to S_1 + qS_2$. Similar to examples \ref{ex1}, \ref{ex2} and \ref{ex3}, system \eqref{example1} possesses besides a positive equilibrium a boundary equilibrium. The exponential convergence towards to positive equilibrium was shown in \cite{FT18} under the assumption that the strong solution exists globally. By applying Theorem \ref{thm1}, we show that this assumption can be removed. Indeed, it is easy to check that \eqref{A3} is satisfied with $K=0$, $h_1(u_1) = u_1$, $h_2(u_2) = u_2$, and $h_3(u_3) = (q+1)u_3$, while \eqref{A4} holds with
	\begin{equation*}
		A = \begin{pmatrix}
			1&0&0\\
			1&1&0\\
			1&1&1
		\end{pmatrix}
	\end{equation*}
	and $r = 1$. Since $r = 1$, it is remarked that \eqref{quasi-uniform} is always satisfied, and consequently the global existence in all dimensions follows from Theorem \ref{thm1}.
	\begin{theorem}
		For any non-negative, bounded initial data, \eqref{example1} has a unique, global strong solution which is bounded uniformly in time, i.e.
		\begin{equation}\label{uniform_Linfty}
			\sup_{i=1,\ldots, 3}\sup_{t\geq 0}\|u_i(t)\|_{L^\infty(\Omega)} < +\infty. 
		\end{equation}
		
		Moreover, if
		\begin{equation*}
			\left\|\frac{1}{u_{2,0}^q}\right\|_{L^\infty(\Omega)} < +\infty,
		\end{equation*}
		then the solution converges exponentially to the positive equilibrium in $L^p(\Omega)$ for all $1\leq p <\infty$, i.e. there exist $C, \lambda_p > 0$ such that
		\begin{equation*}
			\sum_{i=1}^3\|u_i(t) - u_{i,\infty}\|_{L^p(\Omega)} \leq Ce^{-\lambda_p t} \quad \text{for all } \quad t\geq 0.
		\end{equation*}
		Here $u_\infty = (u_{1,\infty}, u_{2,\infty}, u_{3,\infty})$ is the positive equilibrium which solves the following system
		\begin{equation}\label{convergence_Lp}
		\begin{cases}
			u_{3,\infty} - u_{2,\infty}^qu_{1,\infty} = 0,\\
			-u_{3,\infty} + u_{2,\infty}^{q+1} = 0,\\
			u_{1,\infty} + u_{2,\infty} + (q+1)u_{3,\infty} = \frac{1}{|\Omega|}\left(\int_{\Omega}u_{1,0}dx+\int_{\Omega}u_{2,0}dx+(q+1)\int_{\Omega}u_{1,0}dx\right).
		\end{cases}
		\end{equation}
	\end{theorem}
	\begin{proof}
		The global existence and uniform-in-time bound of strong solutions follow immediately from Theorem \ref{thm1} since $r=1$. The exponential convergence to equilibrium in $L^p$-norm \eqref{convergence_Lp} can be obtained similarly to example \ref{ex1} thanks to the following $L^1$-convergence, which was shown in \cite{FT18},
		\begin{equation*}
			\sum_{i=1}^3\|u_i(t) - u_{i,\infty}\|_{L^1(\Omega)} \leq Ce^{-\lambda_1 t} \quad \text{for all } \quad t\geq 0,
		\end{equation*}
		where $C, \lambda_1 > 0$.
	\end{proof}
		
	\medskip
	\par{\bf Acknowledgements:} The second author has been supported by the International Research Training Group IGDK 1754 "Optimization and Numerical Analysis for Partial Differential Equations with Nonsmooth Structures", funded by the German Research Council (DFG)  project number 188264188/GRK1754 and the Austrian Science Fund (FWF) under
	grant number W 1244-N18.

\end{document}